\documentclass[10pt,a4paper,leqno,english,intlimits]{scrartcl}
\usepackage{lmodern}

\usepackage[T1]{fontenc}
\usepackage[utf8]{inputenc}
\usepackage{verbatim}
\usepackage{mathtools}
\usepackage{amsmath}
\usepackage{amsthm}
\usepackage{amssymb}
\usepackage{graphicx}

\makeatletter

\theoremstyle{plain}
\newtheorem{thm}{\protect\theoremname}
\theoremstyle{remark}
\newtheorem{rem}[thm]{\protect\remarkname}
\theoremstyle{remark}
\newtheorem*{rem*}{\protect\remarkname}
\theoremstyle{plain}
\newtheorem{lem}[thm]{\protect\lemmaname}
\theoremstyle{definition}
\newtheorem{defn}[thm]{\protect\definitionname}
\theoremstyle{plain}
\newtheorem{prop}[thm]{\protect\propositionname}
\theoremstyle{plain}
\newtheorem{cor}[thm]{\protect\corollaryname}
\newcommand{\lyxaddress}[1]{
	\par {\raggedright #1
	\vspace{1.4em}
	\noindent\par}
}

\@ifundefined{date}{}{\date{}}
\makeatother

\usepackage{babel}
\providecommand{\corollaryname}{Corollary}
\providecommand{\definitionname}{Definition}
\providecommand{\lemmaname}{Lemma}
\providecommand{\propositionname}{Proposition}
\providecommand{\remarkname}{Remark}
\providecommand{\theoremname}{Theorem}

\begin{document}
\global\long\def\R{\mathbb{R}}%
\global\long\def\graph{\mathop{\mathrm{graph}}}%
\global\long\def\Oh{\mathcal{O}}%
\global\long\def\dist{\mathrm{dist}}%

\title{Smooth convergence to the enveloping cylinder for mean curvature flow
of complete graphical hypersurfaces}
\author{Wolfgang Maurer\thanks{funded by the Deutsche Forschungsgemeinschaft (DFG, German Research Foundation) Project number 336454636}}
\maketitle
\begin{abstract}
For a mean curvature flow of complete graphical hypersurfaces $M_{t}=\graph u(\cdot,t)$
defined over domains $\Omega_{t}$, the enveloping cylinder is $\partial\Omega_{t}\times\R$.
We prove the smooth convergence of $M_{t}-h\,e_{n+1}$ to the enveloping
cylinder under certain circumstances. Moreover, we give examples demonstrating
that there is no uniform curvature bound in terms of the inital curvature
and the geometry of $\Omega_{t}$. Furthermore, we provide an example
where the hypersurface increasingly oscillates towards infinity in
both space and time. It has unbounded curvature at all times and is
not smoothly asymptotic to the enveloping cylinder. We also prove
a relation between the initial spatial asymptotics at the boundary
and the temporal asymptotics of how the surface vanishes to infinity
for certain rates in the case $\Omega_{t}$ are balls.
\end{abstract}

\section{Introduction}

Mean curvature flow is the evolution of a hypersurface that moves
with normal velocity equal to the mean curvature. It is described
by a particularly appealing geometric evolution equation for the embedding
($\frac{\mathrm{d}}{\mathrm{d}t}X=\Delta_{t}X$), while it is also
the negative $L^{2}$-gradient flow of the area functional.

For hypersurfaces $M_{t}$ that are the graphs of function $u(\cdot,t)$,
they solve mean curvature flow if and only if $u$ solves the quasilinear
parabolic partial differential equation
\[
\dot{u}=\left(\delta^{ij}-\frac{u^{i}\,u^{j}}{1+|\nabla u|^{2}}\right)u_{ij}\,.
\]
(We make use of Einstein summation convention.) This equation has
been extensively studied. To name a few of these works, that are in
line with the present article, the author would like to mention \cite{Hui},
where the case of boundary values on a bounded domain is considered,
\cite{EH}, where the long-time existence of solutions for entire
graphs is proven, and \cite{SS}, where complete graphical hypersurfaces
are considered. In the case of complete graphs, the representing functions
$u(\cdot,t)$ are defined on open subsets $\Omega_{t}\subset\R^{n}$
and they diverge to infinity toward the boundary $\partial\Omega_{t}$
such that no boundary values occur and the graphs of $u(\cdot,t)$
are in fact complete hypersurfaces. The domains $\Omega_{t}$ must
be time dependent in this setting and it turns out that they, or respectively
their boundaries $\partial\Omega_{t}$, form a weak solution of mean
curvature flow. This can be understood by the following heuristic.
The hypersurfaces $\graph u(\cdot,t)$ are in some sense asymptotic
to the cylinder $\partial\Omega_{t}\times\R$. Therefore, mean curvature
flow is also expected for $\partial\Omega_{t}\times\R$, and hence
for $\partial\Omega_{t}$ because the $\R$-factor does not contribute
to the mean curvature or the evolution by mean curvature flow.

The main aim of this article is to prove that $\graph u(\cdot,t)$
is \emph{smoothly} asymptotic to the cylinder $\partial\Omega_{t}\times\R$
under certain assumptions. For initially bounded curvature and bounded
$\Omega_{0}$ the asymptotic is smooth for positive times $t>0$ as
long as $\partial\Omega_{t}$ is not singular. It seems to be difficult
to apply the method of proof, which relies on local graphical representations
and pseudolocality, past singularities. Therefore, we investigate
the case of noncollapsed mean curvature flow with a different method
(the curvature bound of \cite{HK}) that allows going beyond singularities.
So for the noncollapsed mean curvature flow of complete graphs the
smooth asymptotics holds generally (still assuming initially bounded
curvature).

If one drops the assumption of an initial curvature bound, the asymptotic
may not be smooth anymore. We provide an interesting example where
the asymptotic is non-smooth at any time and instead the graphical
hypersurfaces infinitely sheets towards the enveloping cylinder. For
the purpose of the construction we explicitely construct a barrier
function which is defined over a shrinking annulus. This barrier allows
for an estimate of $u$ in terms of the value of $u$ on a surrounding
annulus earlier in time. Using this estimate we can prove that there
are solutions $u$ over a shrinking ball $\Omega_{t}$ for which $\dot{u}(0,t)$
oscillates infinitely such that there are sequences $t_{k}\to T$
and $\tau_{k}\to T$ with $\dot{u}(0,t_{k})\to+\infty$ and $\dot{u}(0,\tau_{k})\to-\infty$.
The example shows that complicated kinds of singularities can appear
at infinity for complete graphical hypersurfaces.

The barrier can also be used to prove a relationship between the spatial
asymptotic of $u(x,0)$ as $x\to\partial\Omega_{0}$ and the temporal
asymptotic $u(0,t)$ as $t\to T$ for rotationally symmetric graphs.
In this context it is mandatory to mention \cite{IW,IWZ} where examples
of rotationally symmetric, complete, graphical surfaces with a continuous
spectrum of blow-up rates are constructed. The blow-up rates in their
examples were intimately connected to the spatial asymptotics, too.

The article is organized as follows. Firstly, we will review the mean
curvature flow of complete graphical hypersurfaces, because it is
central to this paper and to set up the notation. In Section \ref{sec curvaturebound},
we prove the smooth asymptotics to the enveloping cylinder up to the
first singularity for initially bounded curvature. A curvature bound
above some height is central to the proof. Despite this curvature
bound above a height, we include in this section a sequence of examples
demonstrating that there is no uniform curvature bound (independent
of height) depending only on the initial curvature bound and the geometry
of $\Omega_{t}$. In Section \ref{sec noncollComplGraph}, we prove
the smooth asymptotics for noncollapsed mean curvature flows of complete
graphs. The section also provides a construction establishing existence
for such flows. Section \ref{sec barrier} is devoted to the barrier
over an annulus and its applications. The appendix contains invaluable
information about normal graphs and noncollapsed mean curvature flow
as well as weak mean curvature flows.

The author thanks O.\ Schnürer for posing the question discussed
here and for his supervision of the authors PhD-thesis, where this
work originates from. The author is also very grateful for the conversations,
stimulations and for O.\ Schnürer's patience.

\section[Mean curvature flow without singularities]{Recapitulation of mean curvature flow without singularities}

\label{sec MCFwS} Since this paper builds on the ideas and results
of ``Mean curvature flow without singularities'' (\cite{SS}), it
is worthwhile to summarize the main points of \cite{SS}. This section
does not contain new results. (We follow \cite{SS}, or \cite{Mau1}
when we diverge from \cite{SS}.) Instead, it helps setting the notation
and allows us to shorten our exposition lateron when similar steps
as here are needed to be taken.

Mean curvature flow without singularities is mainly about the mean
curvature flow of complete graphical hypersurfaces. 
\begin{enumerate}
\item \textbf{Initial Data:} Let $\Omega_{0}\subset\R^{n}$ be open and
let $u_{0}\colon\Omega_{0}\to\R$ be a locally Lipschitz-continuous
function. We assume that there is a continuous extension $\overline{u}_{0}\colon\R^{n}\to\overline{\R}=[-\infty,\infty]$
of $u_{0}$ such that $\{x:-\infty<\overline{u}_{0}(x)<\infty\}=\Omega_{0}$
and $\overline{u}_{0}|_{\Omega_{0}}=u_{0}$ hold.
\item \textbf{Solution Data:} A \emph{mean curvature flow without singularities}
is a pair $(u,\Omega)$ of an relatively open subset $\Omega\subset\R^{n}\times[0,\infty)$
and a continuous function $u\colon\Omega\to\R$. The zero time slice
of $\Omega$ is supposed to be $\Omega_{0}$, in line with a consistent
notation $\Omega_{t}$ for the time slices ($\Omega=\bigcup_{t\ge0}\Omega_{t}\times\{t\}$).
Moreover, we suppose that $u(\cdot,0)=u_{0}$ holds. For this reason,
we call $(u_{0},\Omega_{0})$ the initial data for $(u,\Omega)$.

\textbf{Maximality condition:} We suppose that there exists a continuous
function $\overline{u}\colon\R^{n}\times[0,\infty)\to\overline{\R}$
such that $\{(x,t):-\infty<\overline{u}(x,t)<\infty\}=\Omega$ and
$u|_{\Omega}=u$ hold.

\textbf{Equation:} The function $u$ is supposed to be smooth and
to satisfy the equation of graphical mean curvature flow on $\Omega\setminus(\Omega_{0}\times\{0\})$,
i.e., 
\begin{equation}
\partial_{t}u=\left(\delta^{ij}-\frac{u^{i}\,u^{j}}{1+|Du|^{2}}\right)u_{ij}\;.\label{eq GMCFwS}
\end{equation}

\item \textbf{Hypersurfaces:} We denote by $M_{t}\coloneqq\graph u(\cdot,t)$
the graphical hypersurfaces that move by their mean curvature (locally
in a classical sense). 
\item \textbf{Shadow flow:} The family $(\Omega_{t})_{t\ge0}$ is called
the \emph{shadow flow}.
\end{enumerate}
\begin{thm}
\label{thm MCFwS main} For any such initial data $(u_{0},\Omega_{0})$,
there exists a corresponding mean curvature flow without singularities
$(u,\Omega)$. The shadow flow is a weak solution of mean curvature
flow in dimension $n-1$ in a sense explained below (Remark \ref{rem MCFwS}
(\ref{it MCFwS weak flow})).
\end{thm}

\begin{rem}
\label{rem MCFwS} \,
\begin{enumerate}
\item The maximality condition implies $|u(x,t)|\to\infty$ for $(x,t)\to\partial\Omega$.
($\partial\Omega$ denotes the relative boundary of $\Omega$ in $\R^{n}\times[0,\infty)$.)
In particular, the hypersurfaces $M_{t}$ are complete. Moreover,
the maximality condition implies that the solution is maximal in $t$;
stopping the flow at an arbitrary time may prevent the maximality
condition to hold.

The maximality condition is defined slightly differently in \cite{SS}.
Only positive and proper functions $u$ are considered there. We follow
\cite{Mau1} here, where these assumptions are dropped to some extent
and the maximality condition is adapted accordingly. 
\item Although $M_{t}$ is smooth, the formalism allows for changes of the
topology of $M_{t}$. Singularities of $\partial\Omega_{t}$ may be
interpreted as singularities of $M_{t}$ at infinity. 
\item \label{it MCFwS weak flow} In \cite{SS}, the shadow flow is advertised
as a weak solution. They underpin this by showing that (for their
solution) $(\Omega_{t})_{t\in[0,\infty)}$ coincides with the level-set
flow starting from $\Omega_{0}$ $\mathcal{H}^{n}$-almost everywhere
if the level-set flow is non-fattening.

In \cite{Mau1}, for an arbitrary solution $(u,\Omega)$ the shadow
flow is interpreted as a weak solution in the sense of a domain flow
(cf.\ Definition \ref{def ncMCF domain flow}). 

\end{enumerate}
\end{rem}

\begin{proof}[Discussion of the proof of Theorem \ref{thm MCFwS main}]
 One constructs an approximating sequence of functions $v_{k}\colon\R^{n}\times[0,\infty)\to\overline{\R}$.
It needs to satisfy the following form of local equicontinuity: For
any $a\in\R$, any $R\in\R$, and any $\varepsilon>0$ there is $\delta=\delta(a,R,\varepsilon)>0$
and an index $K=K(a,R,\varepsilon)\in\mathbb{N}$ such that for any
$k\ge K$ and any $(x,s),(y,t)\in\R^{n}\times[0,\infty)$ with $|x|<R$,
$|v_{k}(x,s)|<a$, and $|x-y|+|s-t|<\delta$ we have $|v_{k}(x,s)-v_{k}(y,t)|<\varepsilon$.
A variation on the Arzelà-Ascoli theorem then shows that a subsequence
of $(v_{k})_{k\in\mathbb{N}}$ converges pointwise to a continuous
function $\overline{u}\colon\R^{n}\times[0,\infty)\to\overline{\R}$.
We set $\Omega\coloneqq\{(x,t):-\infty<\overline{u}(x,t)<\infty\}$
and $u\coloneqq\overline{u}|_{\Omega}$. The convergence is locally
uniform on $\Omega$.

To be approximating, the sequence needs to satisfy $v_{k}(\cdot,0)\to\overline{u}_{0}$
pointwise, where $\overline{u}_{0}$ is the above extension of the
initial function such that $\overline{u}(\cdot,0)=\overline{u}_{0}$
holds. Furthermore, for any $a,R\in\R$ and any $t_{0}>0$, we suppose
that there is an index $K=K(a,R,t_{0})$ such that $v_{k}$ is a smooth
solution of (\ref{eq GMCFwS}) on the set $\{(x,t):|v_{k}(x,t)|<a,\,|x|<R,\,t_{0}<|t|<R\}$.
Moreover, we assume uniform estimates of the form $|\partial^{\alpha}v_{k}|\le C(\alpha,a,R,t_{0})$
for $k\ge K$ and for all multi-indices $\alpha$ on this set. The
subsequential convergence $v_{k}\to u$ is then locally smooth on
$\Omega\setminus(\Omega_{0}\times\{0\})$. As a consequence, $u$
solves (\ref{eq GMCFwS}) on $\Omega\setminus(\Omega_{0}\times\{0\})$
and is as asserted.

We have summarized how one obtains a solution from an approximating
sequence and what are sufficient conditions on this sequence. One
still has to find the approximating sequence and prove the local estimates
on the functions and its derivatives. For the approximations one could
either solve initial boundary value problems or use the flow of closed
hypersurfaces which have graphical parts. These are just two options
and one cannot say in general how to approximate; it depends on the
given problem.

For example, we can find an approximating sequence in the following
way. One considers for $a\in\R_{+}$ the functions $\varphi_{a}\colon\overline{\R}\to\R$
with 
\begin{equation}
\varphi_{a}(x)\coloneqq\begin{cases}
x & |x|<a,\\
a & x\ge a,\\
-a & x\le-a.
\end{cases}\label{eq MCFwS phi}
\end{equation}
Then one mollifies $\varphi_{a}\circ\overline{u}$ and restricts to
a ball. Solving graphical mean curvature flow on this ball with this
initial function and holding the boundary values fixed over time,
we find an approximator. (It can be extended to all of $\R^{n}$ by
an arbitrary value.) An approximating sequence is obtained by taking
increasingly larger $a$, finer mollification parameter, and larger
balls.

For the local estimates it is often possible to use the height function
to construct a cut-off function.

We don't say anything about the shadow flow here because the details
of that part of the proof are not important to us.
\end{proof}

\section[Curvature Bounds]{Curvature bounds for mean curvature flow of complete graphs}

\label{sec curvaturebound}

In this section we give hypothesis that render the heuristics rigorous
that for a mean curvature flow without singularities $(u,\Omega)$,
the graph of $u(\cdot,t)$ and $\partial\Omega_{t}\times\R$ behave
alike in sufficient heights. The first section, however, lowers our
expectations towards a curvature bound for mean curvature flow of
complete graphs.

\subsection{Impossibility of a controlled curvature bound}

In the title of this paragraph, one would need to explain what is
meant by ``controlled''. The issue is, what quantities is the curvature
bound allowed to depend on? We will only allow the geometry of $\Omega_{t}$
and the initial curvature here. We cannot exclude curvature bounds
that depend on more elaborate quantities.

\begin{comment}
Graphical mean curvature flow tends to locally smooth out peaks of
curvature: In \cite{EH} a curvature estimate is proven that shows
that $|A|^{2}$ decays like $\frac{1}{t}$ for small times for a graphical
mean curvature flow. This motivates the conjecture that for a mean
curvature flow without singularities $(u,\Omega)$, the curvature
can be estimated for positive times solely in terms of the geometry
of $\Omega_{t}$. However, the example from Section \ref{sec barrier}
already shows that this is not quite true. In this example, the domain
$\Omega$ resembles shrinking balls. So the geometry is explicitly
given and well controlled. Nevertheless, the curvature of $M_{t}$
is unbounded for all times. This example indicates that the curvature
peak $\frac{1}{t}$-decay is not sufficient if the curvature accumulates
too much at infinity.

A modified conjecture could state that the curvature is bounded by
the initial curvature and the geometry of $\Omega_{t}$. However,
this conjecture is also false as we are going to demonstrate now.
\end{comment}

We are going to construct a sequence of examples with uniformly bounded
initial curvature over a domain which is a shrinking ball. Nevertheless,
the curvature in these examples becomes arbitrarily large at some
fixed time before the ball disappears. This demonstrates that there
is no uniform curvature bound for the mean curvature flow without
singularities that depends only on the initial curvature and the geometry
of $\Omega_{t}$.

We start with $\Omega_{0}=B_{1}(0)\subset\R^{2}$. Every solution
$(u,\Omega)$ of the mean curvature without singularities has $\Omega$
resemble the shrinking ball solution, which exists up to the time
$t=\frac{1}{2}$ when $u$ vanishes to infinity. We start with an
initial $u_{0}$ that is sketched in Figure \ref{fig torus}. The
torus shown in that figure is assumed to close at time $t=\frac{1}{4}$
when moved by its mean curvature. The mean curvature flow without
singularities $M_{t}$ stays a smooth graph. So by the avoidance principle
$M_{t}$ must have already passed through the torus by the time the
torus closes. Now consider a sequence of initial surfaces $M_{0}$
that look essentially the same but have their tip further and further
downstairs. These surfaces can be constructed such that their curvature
is uniformly bounded among all of them. But at time $t=\frac{1}{4}$,
the tips must have passed through the torus. For this reason, the
speeds of the tips must become arbitrarily large because they have
to travel increasing distances in the same time. That means the mean
curvature becomes arbitrarily large. We conclude that controlling
the curvature of $M_{0}$ and the geometry of $\Omega_{t}$ is not
sufficient to deduce a curvature bound for $M_{t}$ at later times.

\begin{figure}
\begin{centering}
\includegraphics[height=5.5cm]{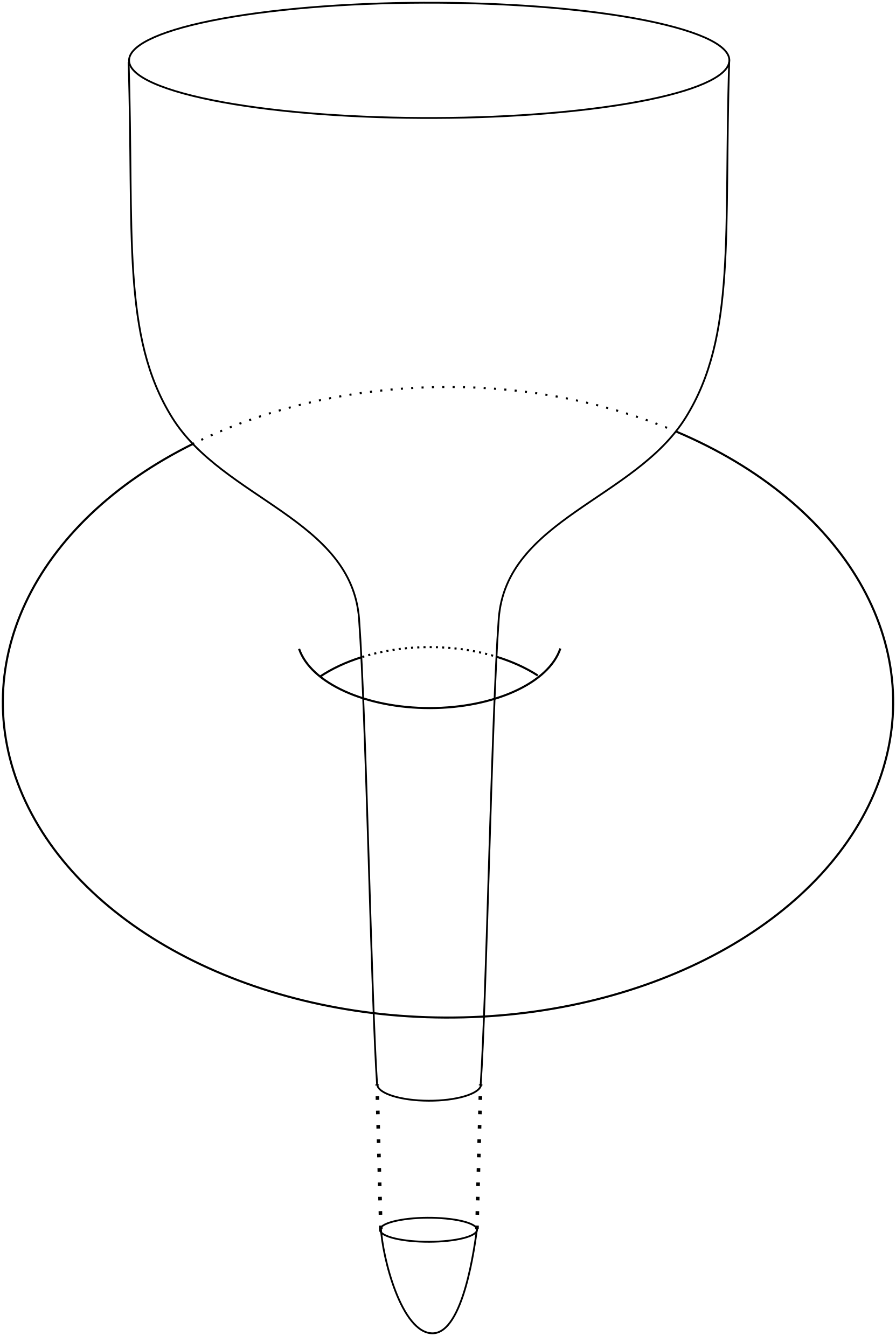}
\par\end{centering}
\caption{The torus closes and the solution must have passed through before.}
\label{fig torus}
\end{figure}

In this counterexample sequence the problem does not arise from infinity.
In fact, there everything is well-behaved. Furthermore, none of the
members of the sequence actually has unbounded curvature. Therefore,
we will keep the hypothesis of initially bounded curvature and of
controlled geometry of $\Omega_{t}$, but we will weaken the assertion.
We will only obtain an \emph{uncontrolled} curvature bound. In accordance
with the fact that in our counterexample sequence the problem does
not come from infinity, we obtain that under those hypothesis the
solution $M_{t}$ is smoothly asymptotic to the cylinder $\Omega_{t}\times\R$.

\subsection{Smooth asymptotics to the cylinder}

For interior curvature estimates we rely on the pseudolocality of
mean curvature flow. This remarkable property of mean curvature flow
means that one has control on solutions by means of local data; the
behavior of the solution far out is irrelevant for that control. This
is not to be confused with locality because a change far out will
have effects everywhere, but these are mild enough such that estimates
can still be derived from local data. This should be contrasted to
the heat equation, where one can realize any given temperature after
any given time at any given point in a room by making the walls of
the room sufficiently hot.

We use the interior curvature estimates from \cite{CY}, but one could
also use \cite{Lah}, which is based on the pseudolocality theorem
in \cite{INS}, in conjunction with \cite{EH}.
\begin{thm}
\label{thm cb Yin} There exist $0<\varepsilon\le1$ and $L>0$ with
the following property. Let $r>0$. Let $(M_{t})_{0\le t\le T}$ be
a smooth solution of mean curvature flow which is properly embedded
in $B_{r}(x_{0})$ with $x_{0}\in M_{0}$ and such that $M_{0}$ is
graphical in $B_{r}(x_{0})$ with gradient bounded by $L$. Then the
second fundamental form is estimated by 
\begin{equation}
|A(x,t)|^{2}\le\frac{1}{t}+(\varepsilon\,r)^{-2}\label{eq Yin 1}
\end{equation}
for $t\in(0,\min\{(\varepsilon r)^{2},T\}]$ and $x\in B_{\varepsilon r}(x_{0})\cap M_{t}$.

If, in addition, the second fundamental form of $M_{0}\cap B_{r}(x_{0})$
is bounded by $r^{-1}$, then 
\begin{equation}
|A(x,t)|^{2}\le(\varepsilon\,r)^{-2}\label{eq Yin 2}
\end{equation}
holds for $t\in[0,\min\{(\varepsilon r)^{2},T\}]$ and $x\in B_{\varepsilon r}(x_{0})\cap M_{t}$.
\end{thm}

\begin{thm}
\label{thm cb main} Let $(u,\Omega)$ be a solution of the mean curvature
flow without singularities. Suppose that $({\partial}{\Omega}_{t})_{0\le t\le T}$
is a smooth and compact solution of the mean curvature flow in dimension
$n-1$ for some $T>0$. Moreover, suppose $M_{0}$ has bounded curvature.

Then $M_{t}-h\,e_{n+1}$ converges for $0<t\le T$ in $C_{\mathrm{loc}}^{\infty}$
to the cylinder $\partial\Omega_{t}\times\R$ for $h\to\infty$. In
particular, $\sup\limits _{x\in M_{t}}|A|^{2}(x,t)$ is bounded for
$0\le t\le T$.
\end{thm}

\begin{rem*}
Instead of bounded curvature, one can assume that $M_{0}$ admits
local graph representations of some radius $r$ with small Lipschitz
constant (cf.\ Definition \ref{def admit graph of radius}). Then
the curvature of $M_{t}$ is bounded for small positive times by Theorem
\ref{thm cb Yin}.
\end{rem*}
\begin{proof}[Proof of Theorem \ref{thm cb main}]
Firstly, we are going to prove a curvature estimate above some height.
The asymptotic behavior is then treated afterwards. The idea for the
curvature estimate is to alternately use Theorem \ref{thm cb Yin},
which provides a curvature estimate from a local graph representation,
and that a curvature estimate ensures graph representations, as we
are going to explain below. These two complement each other perfectly
if done carefully. Essential is that the curvature is below a certain
threshold $C_{A}$ so that we can ensure graph representations. Theorem
\ref{thm cb Yin} then supplies us with a curvature bound a small,
but fixed, time later. If this bound is below $C_{A}$, we can then
repeat the argument. The rough idea is depicted in Figure \ref{fig iteration process}.

\begin{figure}

\begin{centering}
\includegraphics[width=0.8\textwidth]{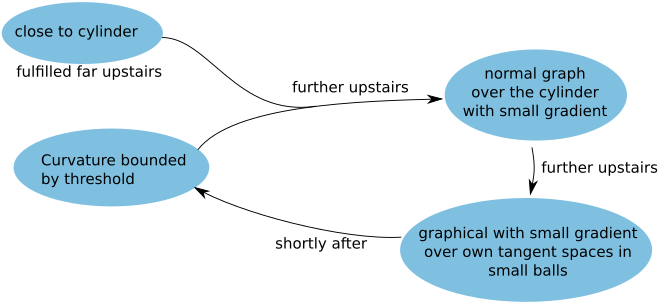}
\par\end{centering}
\caption{The rough idea of the proof.}

\label{fig iteration process}
\end{figure}

We set $Z_{t}\coloneqq\partial\Omega_{t}\times\R$.

Let $0<\varepsilon$, $L\le1$ be as in Theorem \ref{thm cb Yin}.
By Proposition \ref{prop ntg} there is $r>0$ dependent on 
\[
L,\quad\sup_{t\in[0,T]}\sup_{x\in Z_{t}}|A_{Z_{t}}(x,t)|\,,\quad\text{and}\quad\sup_{t\in[0,T]}\sup_{x,y\in Z_{t}}\frac{\dist_{Z_{t}}(x,y)}{|x-y|}
\]
with the following property. If $M$ is any normal graph over $Z_{t}$
with gradient bounded by $L/6$ and which is in a tubular neighborhood
of thickness at most half of the maximal tubular neighborhood thickness,
then $M$ admits local graph representations of radius $r$ with gradient
bounded by $L$. By decreasing $r$ if necessary, we may assume $r<1$
and $\sup_{M_{0}}|A|\le r^{-1}$, where $|A|$ corresponds to $M_{0}$.

We set $C_{A}\coloneqq2(\varepsilon\,r)^{-1}$. This will function
as a threshold for the curvature. By Corollary \ref{cor close hypersurfaces}
there is $\delta>0$ such that for any two hypersurfaces $M'\subset M$,
where $M'$ has buffer $1$ in $M$, with curvature bounded by $C_{A}$
and which lie in a (tubular) neighborhood of $Z_{t}$ of size $\delta$,
can be locally written as a normal graph over $Z_{t}$ with gradient
bounded by $L/8$.

We set 
\[
a\coloneqq\sup\{u(x,t)\colon t\in[0,T],\,x\in\Omega_{t},\,\dist(x,\partial\Omega_{t})\ge\delta\}\;,
\]
such that $M_{t}^{a}\coloneqq M_{t}\cap\{x^{n+1}>a\}$ lies in a tubular
neighborhood of $Z_{t}$ of size $\delta$. By Lemma \ref{lem close hypersurfaces},
$M_{t}^{a+1}$ is a normal graph over $Z_{t}$ (following Definition
\ref{def normal graph}, where it can be a graph over a subset) if
the curvature of $M_{t}^{a}$ is bounded by $C_{A}$. By our choice
of $r$ in the beginning of the proof, this shows that $M_{t}^{a+1}$
admits local graph representations of radius $r$ with gradient bound
$L$ if $M_{t}^{a}$'s curvature is bounded by $C_{A}$.

Because of $\sup_{M_{0}}|A|\le r^{-1}<C_{A}$, $M_{0}^{a+1}$ admits
local graph representations of radius $r$ with gradient bounded by
$L$. The interior estimate (\ref{eq Yin 2}) yields 
\begin{equation}
|A(x,t)|^{2}\le(\varepsilon\,r)^{-2}<(C_{A})^{2}\label{eq cb 2. Art}
\end{equation}
for $t\in[0,\min\{(\varepsilon r)^{2},T\}]$ and $x\in M_{t}^{a+1+r}\subset M_{t}^{a+2}$.

From (\ref{eq Yin 1}) we infer that, in the case that the curvature
of $M_{t}^{a'}$ ($a'\ge a$) is smaller than $C_{A}$ and, therefore,
$M_{t}^{a'+1}$ admits local graph representations of radius $r$,
like it is demonstrated above, the estimate 
\begin{equation}
\big|A\big(x,t+(\varepsilon\,r)^{2}\big)\big|^{2}\le2\,(\varepsilon\,r)^{-2}<(C_{A})^{2}\label{eq cb 1. Art}
\end{equation}
holds for $x\in M_{t+(\varepsilon\,r)^{2}}^{a'+2}$.

Starting from the bound (\ref{eq cb 2. Art}) on the interval $[0,\varepsilon^{2}\,r^{2}]\cap[0,T]$,
the estimate (\ref{eq cb 1. Art}) can be applied iteratively to show
that the curvature of $M_{t}^{a+2N}$ ($N$ is some controlled number
of steps) cannot reach the threshold $C_{A}$ for any $t\in[0,T]$.
The procedure is demonstrated in Figure \ref{fig C_A}. Starting with
the curvature at the brown dot, (\ref{eq cb 2. Art}) establishes
the curvature bound depicted by the brown line. Since the curvature
stays below the threshold $C_{A}$, we can apply the curvature bound
(\ref{eq cb 1. Art}) to obtain the curvature bound in red. In fact,
the picture even shows how to obtain (\ref{eq cb 1. Art}) from (\ref{eq Yin 1}).
Starting from the red curvature bound, which is below the threshold
$C_{A}$, one can again apply (\ref{eq cb 1. Art}) to obtain the
orange curvature bound. In this way, we proceed through all the colors
of the rainbow (and beyond). We will reach the time $T$ in finitely
many steps, providing us with curvature estimates on the whole time
interval. However, the region where the estimates hold go up in height
by 2 units per step.

\unitlength=1cm 
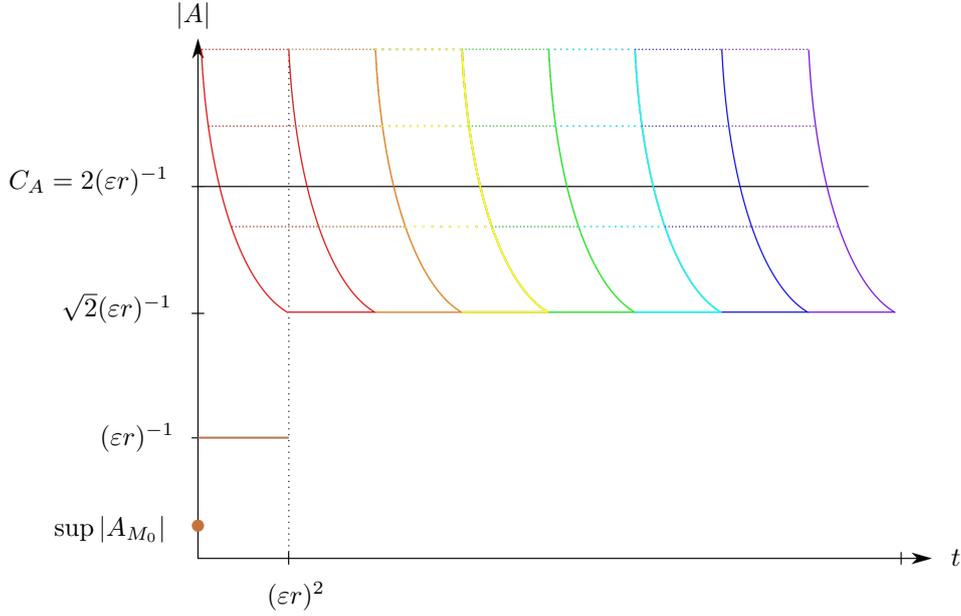
\begin{figure}
\centering \begin{picture}(12,8) \put(2,0.5){\includegraphics[height=7cm]{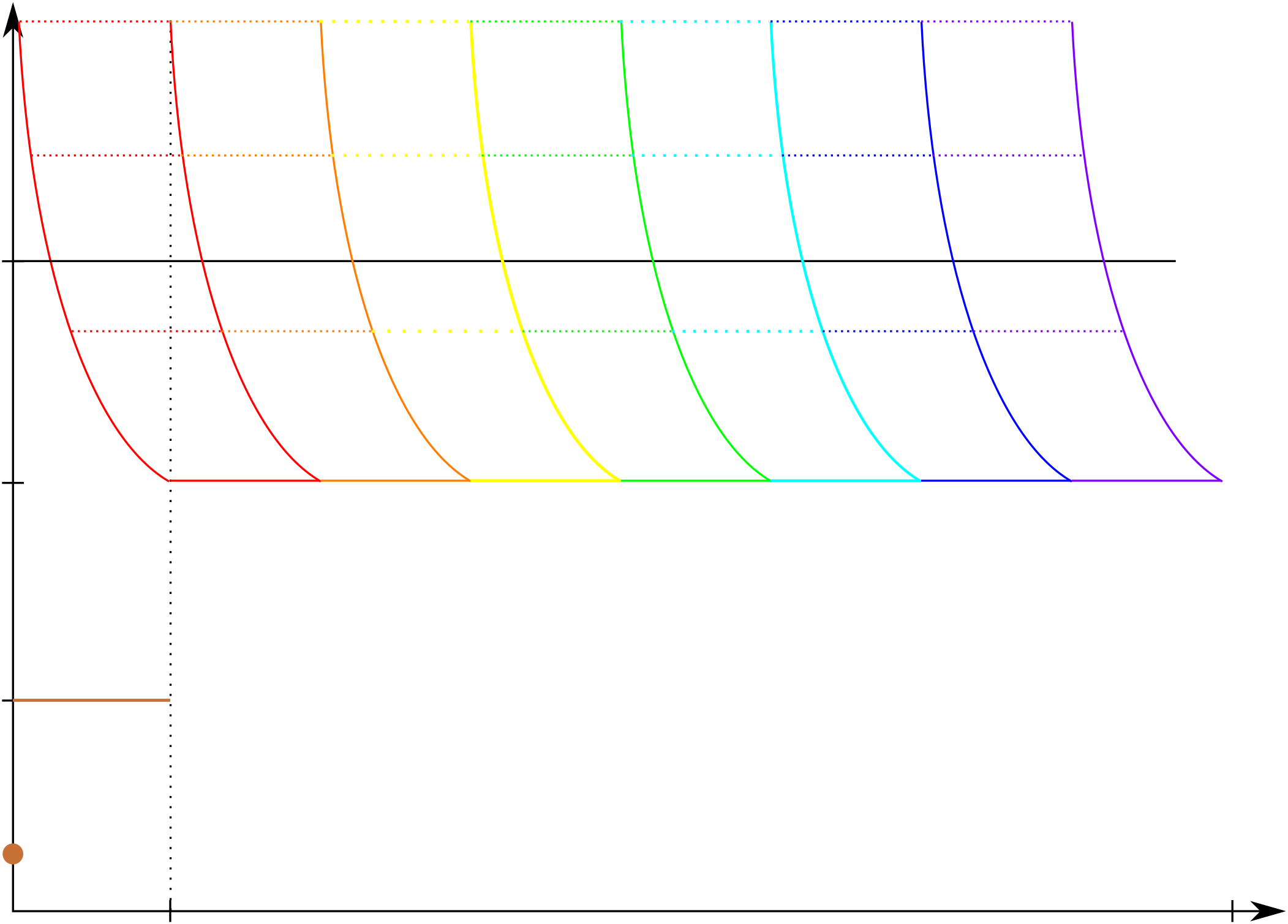}}
\put(0.2,0.9){$\sup|A_{M_{0}}|$} \put(0.8,2.1){$(\varepsilon r)^{-1}$}
\put(3,0){$(\varepsilon r)^{2}$} \put(0.3,3.8){$\sqrt{2}(\varepsilon r)^{-1}$}
\put(-0.4,5.5){$C_{A}=2(\varepsilon r)^{-1}$} \put(1.8,7.7){$|A|$}
\put(12,0.5){$t$} \end{picture} %\end{center}
 \caption{The iterative procedure to obtain the curvature estimate.}
\label{fig C_A}
\end{figure}

Below the height $a+2N$ the curvature is bounded (in a non-controlled
fashion) because it is a continuous function on a compact set.

Now that we have curvature estimates, we can work on improving them:
Because of the curvature bound, in great heights the solution $M_{t}$
is $C^{1}$-close to the cylinder $Z_{t}$. This follows from Lemma
\ref{lem close hypersurfaces}. We translate the solutions downwards,
so we consider $M_{t}-h\,e_{n+1}$, and by the Arzelà-Ascoli-Theorem
we can extract a limit for $h\to\infty$ which must be the enveloping
cylinder $Z_{t}$. For $t>0$, the convergence is not only uniform
($C^{0}$) but by interior estimates for $|\nabla^{k}A|$, that hold
for mean curvature flows of bounded curvature \cite{EH}, and interpolation
inequalities the uniform convergence expands to smooth convergence.
In other words, the solution $M_{t}$ is smoothly asymptotic to the
cylinder $Z_{t}$.
\end{proof}
\begin{rem*}
In the dimensions $n=1,2$, the result holds beyond singularities:
One can obtain the smooth convergence of $M_{t}-h\,e_{n+1}$ to the
cylinder $\partial\Omega_{t}\times\R$ for all positive times $t>0$
provided $\partial\Omega_{t}$ is not singular. We are assuming that
$\Omega_{0}$ is smooth and bounded.

In the case $n=1$, $\Omega_{0}$ is a union of bounded intervals
with fixed distances between them and $\Omega_{0}=\Omega_{t}$ holds
for all times. Thus, there are no singular times and $M_{t}$ is smoothly
asymptotic to $\partial\Omega_{t}\times\R$ for all $t>0$.

In the case $n=2$, the Gage-Hamilton-Grayson theorem \cite{GH,Gra}
ensures that any singularity of $(\partial\Omega_{t})_{t}$ looks
like a shrinking circle when magnified. Therefore, the connected component
of $\partial\Omega_{t}$ on which the singularity appears becomes
extinct in that singularity. The proof of Theorem \ref{thm cb main}
can be carried out on each connected component of $\partial\Omega_{t}$
separately. And none of the connected components have seen singularities
in their past. Loosely speaking, when a singularity arises, all problems
related to it disappear in that singularity.
\end{rem*}

\section[$\alpha$-Noncollapsed MCF Without Singularities]{$\alpha$-noncollapsed mean curvature flow without singularities}

\label{sec noncollComplGraph}

\subsection{Asymptotics to the cylinder}

Let $(u,\Omega)$ be a mean curvature flow without singularities with
$u\ge0$. We note that $u(x,t)\to\infty$ as $(x,t)\to\partial\Omega$.
We assume that $\Omega$ is bounded. In particular, any $\Omega_{t}\subset\R^{n}$
is bounded and $\Omega_{t}=\emptyset$ for all $t\ge T$ with some
$T>0$. Furthermore, we suppose that $\Omega_{0}$ is smooth, mean
convex, and $\alpha$-noncollapsed for some $\alpha>0$ (Definition
\ref{def alpha-nc}). We also assume that $u_{0}=u(\cdot,0)$ is smooth,
that $M_{0}$ has bounded curvature, and that $M_{t}=\graph u|_{\Omega_{t}}$
is mean convex and $\alpha$-noncollapsed for all $t\ge0$ (if $M_{t}\neq\emptyset$).

Our aim in this section is to prove
\begin{thm}
\label{thm ncg asymptotic} With the above assumptions, if $\partial\Omega_{t}\subset\R^{n}$
is a smooth hypersurface for some $t>0$ then $M_{t}-h\,e_{n+1}$
converges in $C_{\textrm{loc}}^{\infty}$ to the cylinder $\partial\Omega_{t}\times\R$
for $h\to\infty$.
\end{thm}

Like in section \ref{sec curvaturebound}, establishing a curvature
bound is crucial. We exploit the $H$-dependent curvature bound from
Corollary \ref{cor alpha-nc curvature bound}. Firstly, we will use
it to demonstrate that the mean curvature $H$ behaves continuously
at infinity and converges to the mean curvature of $\partial\Omega_{t}$
(Theorem \ref{thm H continuous}). Once we have control on $H$, we
can directly use the $H$-dependent estimates to bound all the terms
$|\nabla^{k}A|$, $k\in\mathbb{N}$, for $M_{t}$.
\begin{proof}[Proof of Theorem \ref{thm ncg asymptotic}]
 Let $t>0$ be a time when $\partial\Omega_{t}$ is smooth. Theorem
\ref{thm H continuous} asserts that $H[u](\cdot,t)$ is continuous
with boundary values $H[\partial\Omega_{t}]$ (see Theorem \ref{thm H continuous}
for the terminology). The continuity of $H[u](\cdot,t)$, the compactness
of $\overline{\Omega}_{t}$, and the boundedness of $H[\partial\Omega_{t}]$
together imply the boundedness of $H[u](\cdot,t)$ and therefore the
boundedness of the mean curvature $H$ of $M_{t}$.

Because $M_{t}$ belongs to an $\alpha$-noncollapsed mean curvature
flow, Corollary \ref{cor alpha-nc curvature bound} implies that the
whole geometry of $M_{t}$ is bounded, i.e., all terms $|\nabla^{l}A|$
($l=0,1,\ldots$) are bounded. This applies even for times slightly
before. The smooth convergence now follows from Lemma \ref{lem close hypersurfaces},
the theorem of Arzelà-Ascoli, and interpolation inequalities.
\end{proof}
\begin{thm}
\label{thm H continuous} Let $H[u](\cdot,t)$ be the mean curvature
of the graph of $u(\cdot,t)$ defined as a function on $\Omega_{t}$.
Then $H[u]$ is continuously extendable to $\overline{\Omega}$ with
values in $(0,\infty]$, and there holds $H[u](x,t)=H[\partial\Omega_{t}](x)$
for $x\in\partial\Omega_{t}$ with $H[\partial\Omega_{t}]$ the mean
curvature of the boundary $\partial\Omega_{t}\subset\R^{n}$ (may
be infinite).
\end{thm}

\begin{proof}
\global\long\def\M{\mathcal{M}}%
We denote the \emph{spacetime track} by 
\begin{equation}
\M\coloneqq\bigcup_{t\in[0,\infty)}M_{t}\times\{t\}\subset\R^{n+1}\times[0,\infty)\;.
\end{equation}
In what follows, unspecified geometric quantities refer to $M_{t}$.
Points on $M_{t}$ are tracked in time along the normal direction.

We have 
\begin{equation}
|\nabla H|=|g^{ij}\,\nabla_{k}h_{ij}|\le|g^{ij}|\cdot|\nabla_{k}h_{ij}|=\sqrt{n}\,|\nabla A|
\end{equation}
and 
\begin{equation}
|\partial_{t}H|=\big|\Delta H+|A|^{2}\,H\big|\le|g^{ij}|\cdot|g^{kl}|\cdot|\nabla_{ij}^{2}h_{kl}|+|A|^{2}\cdot|g^{kl}|\cdot|h_{kl}|=n\,|\nabla^{2}A|+\sqrt{n}\,|A|^{3}\;.
\end{equation}
Thus, by Lemma \ref{lem spacetime track} and these inequalities,
$|A_{\M}|^{2}$ is bounded by terms of the form $\frac{|\nabla^{l}A|^{2}}{(1+H^{2})^{l+1}}$
with $l=0,1,2$. By virtue of Corollary \ref{cor alpha-nc curvature bound},
the curvature of the spacetime track $\M$ is uniformly bounded for
$t\ge t_{0}>0$. For times $0<t<t_{0}$, $t_{0}$ sufficiently small,
we may use Theorem \ref{thm cb main} to see that $\M$ has bounded
curvature for these times, too.

By assumption every $M_{t}$ is $\alpha$-noncollapsed and because
$M_{t}$ is not a hyperplane we have $H>0$ everywhere. Because $H$
is the normal speed, this implies that the spacetime track $\M$ is
graphical over $\R^{n+1}$ in the time direction. More precisely,
the domain is $D\coloneqq\{(x,y)\in\R^{n+1}\colon y\ge u_{0}(x)\}$.
The representation function is the time of arrival function $\tau\colon D\to\R$.
By the $\alpha$-noncollapsedness of the $M_{t}$ and because of the
boundedness of interior balls, the mean curvature is uniformly bounded
from below, $H\ge c>0$. This implies that $|\nabla\tau|$ is uniformly
bounded. Together with the curvature bound for $\M=\graph\tau$, this
implies a finite bound for $\|\tau\|_{C^{2}}$.

The sequence of functions $\tau(x,y+h)$ is monotonically increasing
and, by the $C^{2}$-bound, converges for $h\to\infty$ in $C_{\mathrm{loc}}^{1}$
to a $C^{1,1}$ function $\sigma$ defined on $\Omega_{0}\times\R$.
It is independent of the $y$-variable and represents $\partial\Omega\times\R$
as a graph.

The mean curvature depends on the gradient of the time of arrival
function: 
\begin{equation}
H[u](x,t)=\left|\nabla\tau(x,u(x,t))\right|^{-1}\,.
\end{equation}
For $(x,t)\to\partial\Omega$ we have $u(x,t)\to\infty$ and therefore
\begin{equation}
H[u](x,t)=\left|\nabla\tau(x,u(x,t))\right|^{-1}\to\left|\nabla\sigma(x,0)\right|^{-1}=H[\partial\Omega_{t}](x)\;.\qedhere
\end{equation}
\end{proof}
\begin{rem*}
From the proof one can easily see that $\partial\Omega\in C^{1,1}$
holds.
\end{rem*}
\begin{rem*}
It is worth mentioning that, in contrast to $|A_{\M}|$, $|\nabla A_{\M}|$
can not be estimated by terms of the form $\frac{|\nabla^{l}A|^{2}}{(1+H^{2})^{l+1}}$.
Therefore, we do not get higher estimates for the spacetime track
from Corollary \ref{cor alpha-nc curvature bound}.
\end{rem*}
\begin{lem}
\label{lem spacetime track} The squared norm of the second fundamental
form of the spacetime track $\M$ is given by 
\begin{equation}
|A_{\M}|^{2}=\frac{|A|^{2}}{1+H^{2}}+2\,\frac{|\nabla H|^{2}}{(1+H^{2})^{2}}+\frac{|\partial_{t}H|^{2}}{(1+H^{2})^{3}}
\end{equation}
where the geometric quantities on the right hand side correspond to
those of $M_{t}$.
\end{lem}

\begin{proof}
For a point $(x_{0},y_{0},t_{0})\in\M$ ($x_{0}\in\R^{n}$, $y_{0}\in\R$)
we represent $\M$ locally as the graph of a function $v(x,t)$. By
a rotation in $(x,y)$-space ($\R^{n+1}$) we may assume that $\nabla_{x}v(x_{0},t_{0})=0$
holds. Because $(M_{t})_{t}$ is a mean curvature flow, $v$ solves
graphical mean curvature flow, i.e., $\partial_{t}v=\sqrt{1+|\nabla_{x}v|^{2}}\,H$
holds. At the point $(x_{0},t_{0})$ we obtain $\partial_{t}v=H$,
$\nabla_{x}\partial_{t}v=\nabla_{x}H$, and $\partial_{t}^{2}v=\partial_{t}H$.
Before stating computations with these identities at $(x_{0},y_{0},t_{0})$,
let us briefly fix some notation: $\mathrm{D}_{x}f$ is a linear form
and $\nabla_{x}f$ is the gradient of $f$, which is given by $\nabla_{x}f=g^{-1}\cdot(\mathrm{D}_{x}f)^{\mathsf{T}}$.
\begin{align*}
h_{\M} & =\frac{\mathrm{D}_{(x,t)}^{2}v}{\sqrt{1+|\nabla_{(x,t)}v|^{2}}}=\frac{1}{\sqrt{1+|(0,H)|^{2}}}\begin{pmatrix}\mathrm{D}_{x}^{2}v & (\mathrm{D}_{x}\partial_{t}v)^{\mathsf{T}}\\
\mathrm{D}_{x}\partial_{t}v & \partial_{t}^{2}v
\end{pmatrix}\\
 & =\frac{1}{\sqrt{1+H^{2}}}\begin{pmatrix}h & (\mathrm{D}_{x}H)^{\mathsf{T}}\\
\mathrm{D}_{x}H & \partial_{t}H
\end{pmatrix}\\
g_{\M}^{-1} & =\left(\left(\delta^{\alpha\beta}\right)-\frac{\nabla_{(x,t)}v\otimes\nabla_{(x,t)}v}{1+|\nabla_{(x,t)}v|^{2}}\right)=\begin{pmatrix}\left(\delta^{ij}\right) & 0\\
0 & 1-\frac{H^{2}}{1+H^{2}}
\end{pmatrix}=\begin{pmatrix}g^{-1} & 0\\
0 & \frac{1}{1+H^{2}}
\end{pmatrix}\\
A_{\M} & =g_{\M}^{-1}\cdot h_{\M}=\frac{1}{\sqrt{1+H^{2}}}\begin{pmatrix}g^{-1}\cdot h & g^{-1}\cdot(\mathrm{D}_{x}H)^{\mathsf{T}}\\
\frac{\mathrm{D}_{x}H}{1+H^{2}} & \frac{\partial_{t}H}{1+H^{2}}
\end{pmatrix}\\
 & =\frac{1}{\sqrt{1+H^{2}}}\begin{pmatrix}A & \nabla_{x}H\\
\frac{\mathrm{D}_{x}H}{1+H^{2}} & \frac{\partial_{t}H}{1+H^{2}}
\end{pmatrix}\\
|A_{\M}|^{2} & =\textrm{tr}(A_{\M}^{2})=\frac{1}{1+H^{2}}\mathrm{tr}\begin{pmatrix}A^{2}+\nabla_{x}H\cdot\frac{\mathrm{D}_{x}H}{1+H^{2}} & A\cdot\nabla_{x}H+\nabla_{x}H\cdot\frac{\partial_{t}H}{1+H^{2}}\\
\frac{\mathrm{D}_{x}H}{1+H^{2}}\cdot A+\frac{\partial_{t}H}{1+H^{2}}\,\frac{\mathrm{D}_{x}H}{1+H^{2}} & \frac{\mathrm{D}_{x}H}{1+H^{2}}\cdot\nabla_{x}H+\left(\frac{\partial_{t}H}{1+H^{2}}\right)^{2}
\end{pmatrix}\\
 & =\frac{1}{1+H^{2}}\left(|A|^{2}+2\,\frac{|\nabla_{x}H|^{2}}{1+H^{2}}+\frac{(\partial_{t}H)^{2}}{(1+H^{2})^{2}}\right)\;.\qedhere
\end{align*}
\end{proof}

\subsection[Construction]{Construction of an $\alpha$-noncollapsed mean curvature flow without
singularities}

Under the presumption that $\Omega_{0}$ is bounded and smooth, $H[\partial\Omega_{0}]>0$,
$|A_{M_{0}}|\le C$, and $H_{M_{0}}\ge c>0$, we will construct a
mean curvature flow without singularities that is $\alpha$-noncollapsed
for some $\alpha>0$ and matches these data. Theorem \ref{thm ncg asymptotic}
from the last section tells us that this flow is smoothly asymptotic
to the cylinder $\partial\Omega_{t}\times\R$ for all times $t$ when
$\partial\Omega_{t}$ is smooth.

At first, we describe our approach. We approximate $M_{0}$ by closed
hypersurfaces. This is done in a way such that these approximating
hypersurfaces are $\alpha$-noncollapsed. We let flow these by the
level-set flow. By the results of \cite{HK}, the $\alpha$-noncollapsedness
persists along this flow with the same $\alpha$. The level-set flow
will in fact be a smooth graphical flow below a certain height. This
height tends in the limit of the approximation to infinity. We can
then employ a limit process to obtain a mean curvature flow without
singularities which is $\alpha$-noncollapsed.

\paragraph{Approximation of $M_{0}$ by compact $M_{0}^{\delta}$.}

Let $\delta>0$ (to be thought of being small). The initial surface
$M_{0}$ is above a height $a_{\delta}=\max\{u_{0}(x)\colon x\in\Omega_{0},\,\dist(x,\partial\Omega_{0})\ge\delta\}$
in a $\delta$-neighborhood of $N\coloneqq\partial\Omega_{0}\times\R$.
If $\delta$ is sufficiently small, then Lemma \ref{lem close hypersurfaces}
yields that $M_{0}$ is a normal graph over $N$ above the height
$a_{\delta}$ (according to our Definition \ref{def normal graph}
we also speak of a normal graph over $N$ if we have a graph over
a subset of $N$). Let $v$ be the representation function for this
normal graph. We note that we always choose the outward pointing normal.
But because we define the tubular diffeomorphism in the form $(p,d)\mapsto p-d\,\nu$,
the function $v$ is positive (cf.\ Definitions \ref{def tubular neighborhood}
and \ref{def normal graph}). We have $v=\Oh(\delta)$, $|\nabla v|=\Oh(\sqrt{\delta})$,
and $|\nabla^{2}v|\le C$. By Proposition \ref{prop geometry of normal graphs}
there hold 
\begin{align}
g_{M_{0}} & =g_{N}+\Oh(\delta)\\
h_{M_{0}} & =h_{N}+\nabla^{2}v+\Oh(\delta)\;.
\end{align}

We choose a smooth function $\lambda\equiv\lambda_{\delta}\colon\R\to[0,1]$
with $\lambda(x)=0$ for $x\le a_{\delta}$ and $\lambda(x)=1$ for
$x\ge a_{\delta}+\frac{1}{\delta}$. We can engineer this function
in such a way that on $(a_{\delta},a_{\delta}+\frac{1}{\delta})$
we have $\lambda'>0$ and that $|\lambda'|+|\lambda''|^{1/2}=\Oh(\delta)$
holds. By slight abuse of notation, $\lambda$ becomes a function
on $N$ as $\lambda(x^{n+1})$.

On $N\cap\{a_{\delta}<x^{n+1}\le a_{\delta}+\frac{2}{\delta}\}$ we
define for $0<\varepsilon<\delta^{3}$ 
\begin{equation}
w\coloneqq(1-\lambda)\,v+\lambda\,\frac{\varepsilon}{2}\left(a_{\delta}+\frac{2}{\delta}-x^{n+1}\right)^{2}.
\end{equation}
Then we have 
\begin{align}
\begin{split}\nabla_{e_{n+1}}w & =(1-\lambda)\,\nabla_{e_{n+1}}v+\lambda\,\varepsilon\left(x^{n+1}-\left(a_{\delta}+\frac{2}{\delta}\right)\right)\\
 & \quad+\lambda'\left(\frac{\varepsilon}{2}\left(a_{\delta}+\frac{2}{\delta}-x^{n+1}\right)^{2}-v\right),
\end{split}
\label{eq cnmcf w'<0}\\
\nabla_{\partial\Omega_{0}}w & =(1-\lambda)\,\nabla_{\partial\Omega_{0}}v\;,\\
\nabla_{\partial\Omega_{0}}^{2}w & =(1-\lambda)\,\nabla_{\partial\Omega_{0}}^{2}v\;,\\
\nabla_{\partial\Omega_{0}}\nabla_{e_{n+1}}w & =\nabla_{e_{n+1}}\nabla_{\partial\Omega_{0}}w=(1-\lambda)\,\nabla_{\partial\Omega_{0}}\nabla_{e_{n+1}}v-\lambda'\,\nabla_{\partial\Omega_{0}}v\;,\\
\begin{split}\nabla_{e_{n+1}}^{2}w & =(1-\lambda)\,\nabla_{e_{n+1}}^{2}v+\lambda\,\varepsilon+2\,\lambda'\left(\varepsilon\left(x^{n+1}-\left(a_{\delta}+\frac{2}{\delta}\right)\right)-\nabla_{e_{n+1}}v\right)\\
 & \quad+\lambda''\left(\frac{\varepsilon}{2}\left(a_{\delta}+\frac{2}{\delta}-x^{n+1}\right)^{2}-v\right)\;.
\end{split}
\end{align}
If $0<\varepsilon<\delta^{3}$ is sufficiently small (depending on
$v$), we have $0<w<v$, $\nabla w=(1-\lambda)\nabla v+\Oh(\delta)$,
and $\nabla^{2}w=(1-\lambda)\,\nabla^{2}v+\Oh(\delta)$ on $N\cap\{a_{\delta}<x^{n+1}\le a_{\delta}+\frac{2}{\delta}\}$.

Let $W$ be the hypersurface that is given as the normal graph over
$N$ with representation function $w$. Then, again by Proposition
\ref{prop geometry of normal graphs} 
\begin{align}
g_{W} & =g_{N}+\Oh(\delta)\\
\begin{split}h_{W} & =h_{N}+\nabla^{2}w+\Oh(\delta)=h_{N}+(1-\lambda)\,\nabla^{2}v+\Oh(\delta)\\
 & =(1-\lambda)\,h_{M_{0}}+\lambda\,h_{N}+\Oh(\delta)\;.
\end{split}
\end{align}
Therefore, if we adjust the constants a little bit and choose $\delta$
sufficiently small, there hold $|A_{W}|^{2}\le C$ and $H_{W}\ge c>0$.

The compact approximation $M_{0}^{\delta}$ to $M_{0}$ is now constructed
as follows: In the region $\{x^{n+1}\le a_{\delta}\}$, it coincides
with $M_{0}$. In the region $\{a_{\delta}<x^{n+1}\le a_{\delta}+\frac{2}{\delta}\}$,
we set it equal to $W$. From this part we obtain the part above the
height $a_{\delta}+\frac{2}{\delta}$ by reflection at the hyperplane
$\{x^{n+1}=a_{\delta}+\frac{2}{\delta}\}$. By construction, $M_{0}^{\delta}$
is mirror-symmetrical. In fact, its interior is a vain set if $\varepsilon$
is chosen sufficiently small. This stems from the fact that then $\nabla_{e_{n+1}}w<0$
holds, which one can check from (\ref{eq cnmcf w'<0}). Of course,
$M_{0}^{\delta}$ is smooth and we have $|A_{M_{0}^{\delta}}|^{2}\le C$
and $H_{M_{0}^{\delta}}\ge c>0$.

\paragraph{Uniform $\alpha$-noncollapsedness of the $M_{0}^{\delta}$.}
\begin{lem}
\label{lem uniform noncollapsedness} For $C>0$, $c>0$, and $r>0$
there is $\alpha>0$ such that: If a hypersurface $M$ has bounded
second fundamental form $|A|\le C$, mean curvature $H\ge c$, and
admits local graph representations of radius $r$, then $M$ is $\alpha$-noncollapsed.
\end{lem}

\begin{proof}
Let $x\in M$. Then, inside $B_{r}(x)$, $M$ is a graph over $\textrm{T}_{x}M$.
Let $\rho\coloneqq\min\{\frac{r}{2},\,\frac{1}{C}\}$. Then, the two
closed balls $\overline{B}_{\rho}(x\pm\rho\,\nu(x))$ intersect $M$
only in $x$. Thus, if we choose $\alpha\coloneqq c\,\rho$, $x$
admits interior and exterior balls of radius $\frac{\alpha}{H(x)}\le\frac{\alpha}{c}=\rho$.
Because $\alpha$ does not depend on $x$, this demonstrates that
$M$ is $\alpha$-noncollapsed. 
\end{proof}
To apply the lemma it remains to show that the $M_{0}^{\delta}$ admit
local graph representations of a radius $r$ that is independent of
$\delta$. To this end, we realize that the $M_{0}^{\delta}$ are
normal graphs over $\partial\Omega_{0}\times\R$ for the most part,
with two ``$M_{0}$-caps.'' By Proposition \ref{prop ntg} the normal
graph part of each $M_{0}^{\delta}$ admits local graph representations
of a controlled radius independent of $\delta$. Therefore, Lemma
\ref{lem uniform noncollapsedness} is applicable and yields an $\alpha>0$
such that the $M_{0}^{\delta}$ are $\alpha$-noncollapsed.

\paragraph{Evolution of the $M_{0}^{\delta}$ and passing to a limit.}

So far, we have constructed smooth closed hypersurfaces $M_{0}^{\delta}$
such that 
\begin{itemize}
\item $M_{0}^{\delta}$ coincides with $M_{0}$ below the height $a_{\delta}$,
\begin{comment}
\item the second fundamental form of $M_0^{\delta}$ is bounded by a constant $C>0$ independent of ${\delta}$,
\end{comment}
\item %
\begin{comment}
\item the $M_0^{\delta}$ are uniformly mean convex: $H\ge c$ with a constant $c>0$ independent of ${\delta}$,
\end{comment}
\item the $M_{0}^{\delta}$ are mean convex and $\alpha$-noncollapsed with
$\alpha>0$ independent of $\delta$, 
\item The $M_{0}^{\delta}$ are symmetric double graphs with respect to
the hyperplanes $\{x:x^{n+1}=a_{\delta}+\frac{2}{\delta}\}$ in the
sense of \cite{Mau2}.%
\begin{comment}
the interiors $U^{\delta}\subset\R^{n+1}$ of the closed hypersurfaces
$M_{0}^{\delta}$ are vain sets with respect to the hyperplane $\{x^{n+1}=a_{\delta}+\frac{2}{\delta}\}$. 
\end{comment}
\end{itemize}
We can flow $M_{0}^{\delta}$ by the mean curvature flow as in \cite{Mau2}:
Singularities only occur on the hyperplane $\{x:x^{n+1}=a_{\delta}+\frac{2}{\delta}\}$
and the surfaces stay symmetric double graphs. In particular, the
part below the hyperplane is smooth and graphical.

By Proposition \ref{prop ncM uniqueness}, the weak solution is unique.
In particular, $(M_{t}^{\delta})_{t}$ coincides with the level-set
flow. By Theorem \ref{thm ncMCF level-set flow alpha-noncollapsed},
the level-set flow is $\alpha$-non\-col\-lapsed for all $t\ge0$
with the same $\alpha$ as for $M_{0}^{\delta}$. Hence, $\left(M_{t}^{\delta}\right)_{t\in[0,\infty)}$
is $\alpha$-noncollapsed.

We now may use the a priori estimates and the Arzelà-Ascoli argument
as employed before in Section \ref{sec MCFwS} to pass to a limit
as we let $\delta\to0$. The limit is a smooth mean curvature flow
$(M_{t})_{t}$ starting from $M_{0}$ that is graphical and provides
us with a mean curvature flow without singularities $(u,\Omega)$
with initial value $(u_{0},\Omega_{0})$. Clearly, the approximation
outlined in Section \ref{sec MCFwS} would have given us this result
as well. But importantly, the uniform $\alpha$-noncollapsedness of
the approximators carries over to the limit because $M_{t}^{\delta}\to M_{t}$
locally smoothly, proving that $M_{t}$ is $\alpha$-noncollapsed.

Thus, we have proven the following theorem.
\begin{thm}
Let $\Omega_{0}\subset\R^{n}$ be a bounded, smooth, and mean convex
domain. Let $u_{0}\colon\Omega_{0}\to\R$ be a positive and smooth
function such that $u_{0}(x)\to\infty$ for $x\to\partial\Omega_{0}$.
Let $M_{0}\coloneqq\graph u_{0}$. We assume that there are constants
$C,c>0$ such that $|A_{M_{0}}|\le C$ and $H_{M_{0}}\ge c$. Then
there exists a mean curvature flow without singularities $(u,\Omega)$
with initial values $(u_{0},\Omega_{0})$ such that $(M_{t})_{t\in[0,\infty)}$
is $\alpha$-noncollapsed for some $\alpha>0$.
\end{thm}

\section{Barrier over annuli}

\label{sec barrier} In the present section we assume $n\ge2$.

Central to this section is a barrier which is defined over annuli.
This barrier enables us to prove estimates for a mean curvature flow
in terms of the height over an annulus earlier in time. This will
be exploited to obtain various results.

\begin{figure}[h]
\label{fig annulus barrier}
\begin{centering}
\includegraphics[height=3cm]{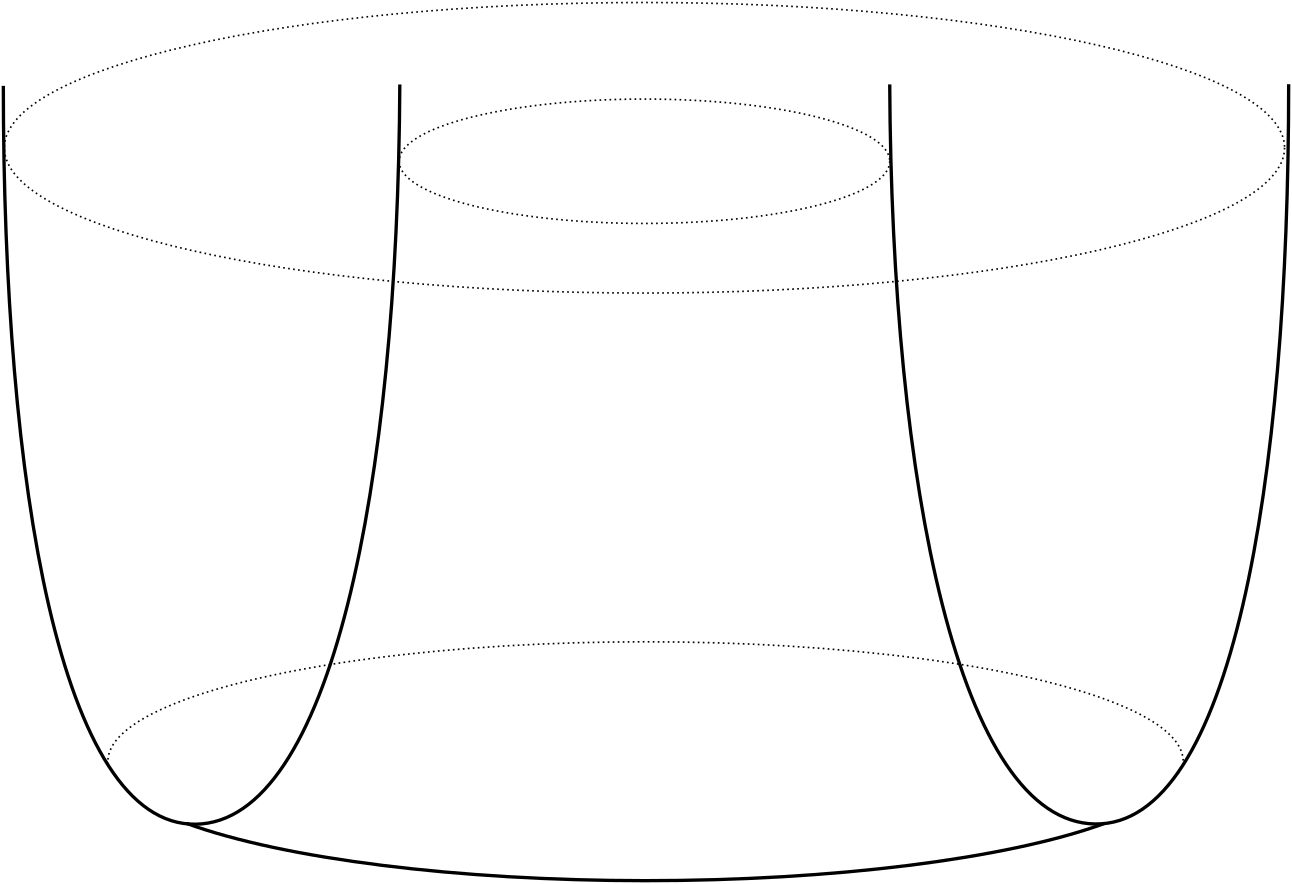}
\par\end{centering}
\caption{Sketch of the barrier.}
\end{figure}

The upper barrier has the following geometry (Fig.~\ref{fig annulus barrier}.1).
Initially start with an annulus in $n$-dimensional Euclidean space.
Over this, we consider a special function which tends to infinity
at the boundary of the annulus. As time goes by, the initial annulus
shrinks by mean curvature flow and the function is adjusted accordingly
to have the shrinking annulus as its domain; while at the same time,
the function is shifted upwards.

In the first section, we will write down the barrier and prove that
it really functions as a barrier. The upshot is Theorem \ref{thm annulus estimate}.
Afterwards, in the second section, we will utilize the barrier to
construct an ``ugly'' example of mean curvature flow without singularities
with wild oscillations which persist up to the vanishing time. As
another application of the barrier, we will prove in the third section
a certain relationship between the spatial asymptotics of a mean curvature
flow without singularities and its temporal asymptotics at the vanishing
time.

\subsection{Construction of the barrier}

The barrier construction starts with a hypersurface that is obtained
from a grim reaper curve which is rotated around the $x^{n+1}$-axis.
The grim reaper is a well-known special solution of curve-shortening
flow (one-dimensional mean curvature flow). It has the explicit graphical
representation ($\lambda>0$) 
\begin{equation}
u_{\mathrm{gr}}(x,t)=-\frac{1}{\lambda}\log\cos(\lambda\,x)+\lambda\,t\qquad\text{for }x\in\left(-\frac{\pi}{2\lambda},\frac{\pi}{2\lambda}\right),\;t\in\R\;.
\end{equation}
In this form $\graph u_{\mathrm{gr}}$ translates with speed $\lambda$
upwards.

Our barrier construction is motivated by this solution and initially
we start with a grim reaper curve rotated around a vertical axis.
Let us have a look at the barrier function now: 
\global\long\def\rad{\sqrt{|x|^{2}+2(n-1)t}}%
 
\begin{equation}
w(x,t)=-\frac{1}{\lambda}\log\cos\left[\lambda\left(\rad-R\right)\right]+\lambda\,t+f(t)\;,%+\left(\lambda+\sqrt{\frac{2\lambda}{\piR}}\right)t,
\end{equation}
where $R,\lambda>0$. As can be seen, $w$ is similar to $u_{\mathrm{gr}}$.
But $x$ is replaced by the term $\rad$ which is motivated by the
radius of an $(n-1)$-dimensional sphere flowing by its mean curvature.
The function $f(t)$ will be adequately chosen to make $w$ satisfy
the correct differential inequality. We will give two possible choices
for $f$, one rather simple, the other more elaborate but much smaller
as $\lambda$ becomes large. However, we include the second one only
for the sake of completeness because in our applications $\lambda\,t$
dominates $f(t)$ for large $\lambda$ in either case.

Before we start proving the differential inequality, we must talk
about the domain of definition of $w$. We set 
\begin{align}
 & t\in I(R)\coloneqq[0,T(R)]\coloneqq\left[0,\frac{R^{2}}{2(n-1)}\right]\\
 & x\in A_{t}(R,\lambda)\coloneqq\left\{ x\in\R^{n}\colon\left(R-\frac{\pi}{2\lambda}\right)<\rad<\left(R+\frac{\pi}{2\lambda}\right)\right\} .\label{eq def A_t}
\end{align}
$A_{t}(R,\lambda)$ is an annulus or a ball, depending on $t$, $R$,
and $\lambda$.

\paragraph{Differential inequality.}

\global\long\def\eck{[\;\cdot\;]}%
 
\global\long\def\sqq{\sqrt{\rule{0em}{1ex}\quad}}%
 To keep the computations accessible, we will make the abbreviations
$\sqq$ for $\rad$ and $\eck$ for $\lambda\left(\sqq-R\right)$.

To show that $w$ is an upper barrier, we must prove 
\begin{equation}
-\dot{w}+\left(\delta^{ij}-\frac{w^{i}w^{j}}{1+|\nabla w|^{2}}\right)w_{ij}\stackrel{!}{\le}0\;.\label{eq differential inequality}
\end{equation}
First, we determine the derivatives. 
\begin{align}
\dot{w} & =\tan\eck\,\frac{n-1}{\sqq}+\lambda+f'(t)\;,\\
w_{i} & =\tan\eck\,\frac{x_{i}}{\sqq}\;,\\
w_{ij} & =\lambda\left(1+\tan^{2}\eck\right)\frac{x_{i}\,x_{j}}{\sqq^{2}}+\frac{\tan\eck}{\sqq}\left(\delta_{ij}-\frac{x_{i}\,x_{j}}{\sqq^{2}}\right).
\end{align}
Substituting into (\ref{eq differential inequality}) yields 
\begin{equation}
\begin{split} & -\dot{w}+\left(\delta^{ij}-\frac{w^{i}w^{j}}{1+|\nabla w|^{2}}\right)w_{ij}\\
 & \quad=-\dot{w}+\frac{1}{1+|\nabla w|^{2}}\bigg((1+|\nabla w|^{2})\delta^{ij}-w^{i}w^{j}\bigg)w_{ij}\\
 & \quad=-\lambda-f'(t)-\tan\eck\,\frac{n-1}{\sqq}\\
 & \qquad+\frac{\lambda\,(1+\tan^{2}\eck)}{1+\tan^{2}\eck\frac{|x|^{2}}{\sqq^{2}}}\left(\left(1+\tan^{2}\eck\frac{|x|^{2}}{\sqq^{2}}\right)\frac{|x|^{2}}{\sqq^{2}}-\tan^{2}\eck\frac{|x|^{4}}{\sqq^{4}}\right)\\
 & \qquad+\frac{\frac{\tan\eck}{\sqq}}{1+\tan^{2}\eck\frac{|x|^{2}}{\sqq^{2}}}\left(\left(1+\tan^{2}\eck\frac{|x|^{2}}{\sqq^{2}}\right)\left(n-\frac{|x|^{2}}{\sqq^{2}}\right)\right.\\
 & \hspace{4.5cm}\left.-\tan^{2}\eck\left(\frac{|x|^{2}}{\sqq^{2}}-\frac{|x|^{4}}{\sqq^{4}}\right)\right)\\
 & \quad=-\lambda-f'(t)-\frac{\tan\eck}{\sqq}(n-1)+\frac{\lambda\,(1+\tan^{2}\eck)}{1+\tan^{2}\eck\frac{|x|^{2}}{\sqq^{2}}}\left(\frac{|x|^{2}}{\sqq^{2}}\right)\\
 & \qquad+\frac{\frac{\tan\eck}{\sqq}}{1+\tan^{2}\eck\frac{|x|^{2}}{\sqq^{2}}}\left(\left(1+\tan^{2}\eck\frac{|x|^{2}}{\sqq^{2}}\right)(n-1)+\left(1-\frac{|x|^{2}}{\sqq^{2}}\right)\right)\\
 & \quad=-f'(t)+\frac{\left(1-\frac{|x|^{2}}{\sqq^{2}}\right)}{1+\tan^{2}\eck\frac{|x|^{2}}{\sqq^{2}}}\left(\frac{\tan\eck}{\sqq}-\lambda\right)\stackrel{!}{\le}0\;.
\end{split}
\end{equation}
In the case $\frac{\tan\eck}{\sqq}\le\lambda$, the inequality clearly
holds. Thus, let us from now on assume $\frac{\tan\eck}{\sqq}>\lambda$.

Neglecting terms with the appropriate signs, we aim to show 
\begin{equation}
\frac{\frac{\tan\eck}{\sqq}}{1+\frac{\tan^{2}\eck}{\sqq^{2}}|x|^{2}}\stackrel{!}{\le}f'(t)\;.\label{eq aim}
\end{equation}
If we can choose $f$ in such a way then the differential inequality
(\ref{eq differential inequality}) follows.

We are going to utilize the following lemma.
\begin{lem}
\label{lem A} Let $J\subset\R_{>0}$ be an interval and $y\colon J\to\R_{>0}$
a monotonically increasing function, which is continuous and surjective.
Then 
\begin{equation}
\frac{y(r)}{1+y^{2}(r)\,r^{2}}\le\frac{1}{a}
\end{equation}
holds for all $r\in J$, where $a$ is the unique solution of $y(a)=\frac{1}{a}$.
\end{lem}

\begin{proof}
It is not hard to see that there is a unique solution $a$ of $y(a)=\frac{1}{a}$
given the hypothesis on $y$.

We distinguish two cases. If $r\le a$ holds, we find 
\begin{equation}
\frac{y(r)}{1+y^{2}(r)\,r^{2}}\le y(r)\le y(a)=\frac{1}{a}\;.
\end{equation}
In the case that $r>a$ holds, we have 
\begin{equation}
\frac{y(r)}{1+y^{2}(r)\,r^{2}}\le\frac{y(r)}{1+y^{2}(r)\,a^{2}}\;.\label{eq lem A 1}
\end{equation}
The expression $\frac{y}{1+y^{2}a^{2}}$ tends to $0$ for both $y\to0$
and $y\to\infty$. Hence, it attains its maximum at an interior point
$b$. We can determine $b$ from the extremality condition 
\begin{equation}
0=\left.\frac{\mathrm{d}}{\mathrm{d}y}\left(\frac{y}{1+y^{2}a^{2}}\right)\right|_{y=b}=\frac{1+b^{2}a^{2}-2b^{2}a^{2}}{(1+b^{2}a^{2})^{2}}=\frac{1-b^{2}a^{2}}{(1+b^{2}a^{2})^{2}}\;.
\end{equation}
We infer that $b=\frac{1}{a}$ is the maximum point. That leads to
\begin{equation}
\frac{y}{1+y^{2}a^{2}}\le\frac{b}{1+b^{2}a^{2}}=\frac{1/a}{1+1}<\frac{1}{a}\label{eq lem A 2}
\end{equation}
for any $y\in\R_{>0}$. Bringing together (\ref{eq lem A 1}) and
(\ref{eq lem A 2}), the assertion follows.
\end{proof}

\paragraph{Application of the lemma.}

We shall derive (\ref{eq aim}) with the help of Lemma \ref{lem A}.
Clearly, $\frac{\tan\eck}{\sqq}$, seen as a function of $r\coloneqq|x|$,
will take the role of $y(r)$. Notice that we fix $t\in I(R)$ here.
The interval $J$ is appropriately chosen to be 
\begin{equation}
J=\left(\sqrt{R^{2}-2(n-1)t},\sqrt{\left(R+\frac{\pi}{2\lambda}\right)^{2}-2(n-1)t}\right),
\end{equation}
such that $\eck\coloneqq\lambda(\sqq-R)\in(0,\frac{\pi}{2})$. This
makes $y(r)=\frac{\tan\eck}{\sqq}$ well defined on $J$ and surjective
onto $\R_{>0}$.

Now we show, that our choice of $y$ is monotonically increasing.
The derivative with respect to $\sqq$ is 
\begin{equation}
\begin{split}\frac{\mathrm{d}}{\mathrm{d}\sqq}\frac{\tan[\lambda(\sqq-R)]}{\sqq} & =\frac{\lambda}{\sqq\,\cos^{2}\eck}-\frac{\tan\eck}{\sqq^{2}}=\frac{\lambda\sqq-\sin\eck\cos\eck}{\sqq^{2}\cos^{2}\eck}\\
 & \ge\frac{\lambda\sqq-\eck%\lambda(\sqq-R)
}{\sqq^{2}\cos^{2}\eck}=\frac{\lambda\,R}{\sqq^{2}\cos^{2}\eck}%
>0\;.
\end{split}
\end{equation}
So $\frac{\tan\eck}{\sqq}$ is increasing as a function of $\sqq$.
And because $\sqq$ is increasing in $r$, $y(r)=\frac{\tan\eck}{\sqq}$
is increasing in $r$ as well.

Finally, Lemma \ref{lem A} is applicable and yields 
\begin{equation}
\frac{\frac{\tan\eck}{\sqq}}{1+\frac{\tan^{2}\eck}{\sqq^{2}}\,|x|^{2}}\le\frac{1}{a_{t}}\;,\label{eq lem A yields}
\end{equation}
where $a_{t}$ is the unique solution in $J$ of the equation 
\global\long\def\rada{\sqrt{a_{t}^{2}+2(n-1)t}}%
 
\begin{equation}
\frac{\tan\left[\lambda\left(\rada-R\right)\right]}{\rada}=\frac{1}{a_{t}}\;.\label{eq a_t sol of}
\end{equation}
To further estimate $\frac{1}{a_{t}}$ in (\ref{eq lem A yields}),
we must extract information from (\ref{eq a_t sol of}). We will now
demonstrate two possible ways to do so.

\paragraph{First choice for $f$.}

We notice that 
\begin{equation}
\tan\left(\lambda\left(\rada-R\right)\right)=\frac{\rada}{a_{t}}\ge1=\tan\left(\frac{\pi}{4}\right)\,.\label{eq annulus pi/4}
\end{equation}
Consequently, and by $2(n-1)t\le R^{2}$, 
\begin{equation}
a_{t}^{2}\ge\left(\frac{\pi}{4\lambda}+R\right)^{2}-2(n-1)t\ge\frac{\pi R}{2\lambda}\;.\label{eq 1. estimate of a_t}
\end{equation}
Choosing 
\begin{equation}
\boxed{f(t)\coloneqq\sqrt{\frac{2\lambda}{\pi R}}\;t}\label{eq 1. choice of f}
\end{equation}
we infer from (\ref{eq lem A yields}) and (\ref{eq 1. estimate of a_t})
\begin{equation}
\frac{\frac{\tan\eck}{\sqq}}{1+\frac{\tan^{2}\eck}{\sqq^{2}}\,|x|^{2}}\le\frac{1}{a_{t}}\le\sqrt{\frac{2\lambda}{\pi R}}=f'(t)\;.
\end{equation}
I.e., (\ref{eq aim}), and hence (\ref{eq differential inequality})
hold.

\paragraph{Second choice for $f$.}

We write 
\begin{equation}
s\coloneqq\sqrt{a_{t}^{2}+2(n-1)t}\;.\label{eq def s}
\end{equation}
We note that $s$ depends on $t$ and lies in the range $\left[R+\tfrac{\pi}{4\lambda},R+\tfrac{\pi}{2\lambda}\right)$,
which can be seen from (\ref{eq annulus pi/4}). With $s$ at hand
we can rewrite (\ref{eq a_t sol of}) as 
\begin{equation}
\frac{1}{a_{t}}=\frac{1}{\sqrt{s^{2}-2(n-1)t}}=\frac{\tan(\lambda\,(s-R))}{s}\;.\label{eq for a_t with s}
\end{equation}
This can easily be solved for $2(n-1)t$. We write $\tau(s)$ for
$2(n-1)t$ viewed as a function of $s$: 
\begin{equation}
\tau(s)\coloneqq s^{2}\left(1-\frac{1}{\tan^{2}(\lambda(s-R))}\right),\qquad s\in\left[R+\tfrac{\pi}{4\lambda},\,R+\tfrac{\pi}{2\lambda}\right)\;.
\end{equation}
We continuously extend $\tau$ to the closed interval through $\tau(R+\frac{\pi}{2\lambda})=(R+\frac{\pi}{2\lambda})^{2}$.
We will estimate $\tau(s)$ from below by 
\begin{equation}
\sigma(s)\coloneqq\left(1-4\lambda\left(R+\frac{\pi}{2\lambda}\right)\right)s^{2}+4\lambda s^{3}\;;
\end{equation}
$\sigma$ is the solution of $\sigma'(s)=\frac{2}{s}\sigma(s)+4\lambda s^{2}$
with $\sigma\left(R+\tfrac{\pi}{2\lambda}\right)=\tau\left(R+\tfrac{\pi}{2\lambda}\right)=\left(R+\tfrac{\pi}{2\lambda}\right)^{2}$.
The function $\tau$ fulfills the differential inequality 
\begin{equation}
\begin{split}\tau'(s) & =2s\left(1-\frac{1}{\tan^{2}(\lambda\,(s-R))}\right)+2\lambda\,s^{2}\frac{1+\tan^{2}(\lambda\,(s-R))}{\tan^{3}(\lambda\,(s-R))}\\
 & \le\frac{2}{s}\tau(s)+4\lambda\,s^{2}\;.
\end{split}
\end{equation}
It follows that 
\begin{equation}
\frac{\mathrm{d}}{\mathrm{d}s}(\tau(s)-\sigma(s))\le\frac{2}{s}\tau(s)+4\lambda\,s^{2}-\left(\frac{2}{s}\sigma(s)+4\lambda\,s^{2}\right)=\frac{2}{s}(\tau(s)-\sigma(s))\;.
\end{equation}
In fact the inequality is strict except for $s=R+\frac{\pi}{4\lambda}$.
Thus, whenever $\tau(s)=\sigma(s)$ for some $s\in(R+\frac{\pi}{4\lambda},R+\frac{\pi}{2\lambda})$,
we have $\tau'(s)<\sigma'(s)$ and there is an $\varepsilon>0$ such
that $\tau>\sigma$ on $(s-\varepsilon,s)$. Considering $\sup\{s\colon\tau(s)<\sigma(s)\}$,
it is now easy to show that $\tau\ge\sigma$ on the whole interval
$\left[R+\frac{\pi}{4\lambda},\,R+\frac{\pi}{2\lambda}\right]$. So
we have shown for $s\in\left[R+\frac{\pi}{4\lambda},\,R+\frac{\pi}{2\lambda}\right]$
\begin{align}
\tau(s) & \ge\left(1-4\lambda\left(R+\frac{\pi}{2\lambda}\right)\right)s^{2}+4\lambda s^{3}\;,\\
\frac{\tau(s)}{\left(R+\frac{\pi}{4\lambda}\right)^{2}} & \ge\frac{\tau(s)}{s^{2}}\ge1-4\lambda\left(R+\frac{\pi}{2\lambda}\right)+4\lambda\,s\;,\\
\begin{split}s & \le R+\frac{\pi}{2\lambda}-\frac{1}{4\lambda}\left(1-\frac{\tau(s)}{\left(R+\frac{\pi}{4\lambda}\right)^{2}}\right)\\
 & =R+\frac{1}{\lambda}\left(\frac{\pi}{2}-\frac{1}{4}\left(1-\frac{\tau(s)}{\left(R+\frac{\pi}{4\lambda}\right)^{2}}\right)\right).
\end{split}
\label{eq upper bound for s}
\end{align}

Let us return to (\ref{eq def s}), the value for $s$ we are interested
in. Let us also remember that $\tau(s)=2(n-1)t$ holds in this context.
With $s\ge R+\frac{\pi}{4\lambda}$ and (\ref{eq upper bound for s})
we can continue from (\ref{eq for a_t with s})
\begin{equation}
\frac{1}{a_{t}}\le\frac{\tan(\lambda(s-R))}{s}\le\frac{1}{R+\frac{\pi}{4\lambda}}\,\tan\left(\frac{\pi}{2}-\frac{1}{4}\left(1-\frac{2(n-1)t}{\left(R+\frac{\pi}{4\lambda}\right)^{2}}\right)\right).\label{eq 2. estimate for a_t}
\end{equation}
We integrate the last term: 
\begin{equation}
\boxed{f(t)\coloneqq\frac{2\left(R+\frac{\pi}{4\lambda}\right)}{n-1}\left[-\log\sin\left(\frac{1}{4}\left(1-\frac{2(n-1)t}{\left(R+\frac{\pi}{4\lambda}\right)^{2}}\right)\right)+\log\sin\left(\frac{1}{4}\right)\right].}\label{eq 2. choice of f}
\end{equation}
Then we obtain (\ref{eq aim}) from (\ref{eq 2. estimate for a_t})
and (\ref{eq lem A yields}). So the differential inequality (\ref{eq differential inequality})
for $w$ follows with this choice for $f$.

\paragraph{Asymptotics of $f$ for large $\lambda$.}

The merit of the second choice for $f$ is the milder growth in $\lambda$.
While (\ref{eq 1. choice of f}) clearly grows like $\sqrt{\lambda}$,
(\ref{eq 2. choice of f}) only grows like $\log(\lambda)$. To see
this let us evaluate $f$ at the final time $T(R)$ of the time interval
under consideration ($T(R)\coloneqq\frac{R^{2}}{2(n-1)}$). Let $f_{1}$
be the first choice, (\ref{eq 1. choice of f}), and let $f_{2}$
be the second, (\ref{eq 2. choice of f}). Then we have 
\begin{equation}
f_{1}(T(R))=\sqrt{\frac{2\lambda}{\pi\,R}}\;\frac{R^{2}}{2(n-1)}=\frac{R}{4(n-1)}\sqrt{\frac{8R\,\lambda}{\pi}}\;.
\end{equation}
On the other hand, for $\lambda\to\infty$: 
\global\long\def\Riv{R+\frac{\pi}{4\lambda}}%
 
\begin{equation}
\begin{split}f_{2}(T(R)) & =\frac{2\left(\Riv\right)}{n-1}\left[-\log\sin\left(\frac{1}{4}\frac{\left(\Riv\right)^{2}-R^{2}}{\left(\Riv\right)^{2}}\right)+\log\sin(\frac{1}{4})\right]\\
 & \simeq-\frac{2R}{n-1}\;\log\sin\left(\frac{1}{4}\frac{\frac{\pi}{2\lambda}R}{R^{2}}\right)\\
 & \simeq-\frac{2R}{n-1}\;\log\left(\frac{\pi}{8R\,\lambda}\right)\\
 & =\frac{2R}{n-1}\;\log\left(\frac{8R\,\lambda}{\pi}\right)\;.
\end{split}
\end{equation}

\paragraph{Barrier property and annulus dependent estimate.}

So far we have proven that for any choice of $f$ ((\ref{eq 1. choice of f})
or (\ref{eq 2. choice of f})) $w$ fulfills the differential inequality
(\ref{eq differential inequality}) for all $t\in I(R)$ and $x\in A_{t}(R,\lambda)\setminus\{0\}$.
For this reason, a mean curvature flow cannot touch $\graph w(\cdot,t)$
from below at any interior point if the flow is disjoint initially.
A special case, however, is the point $(0,w(0,t))$ (if defined) because
$w$ does not solve (\ref{eq differential inequality}) at $x=0$.
In fact, $w$ is not even differentiable at $x=0$. But still, no
smooth surface can touch $\graph w$ from below at $(0,w(0,t))$:
All of the directional derivatives of $w$ at zero are defined and
they all are strictly negative, and hence there is not even a smooth
curve, let alone a hypersurface, that can touch $w$ from below at
that point.

It remains to exclude that $\graph w(\cdot,t)$ is crossed at infinity.
Then $\graph w(\cdot,t)$ is a barrier and we obtain the following
annulus dependent estimate. 
\begin{thm}
\label{thm annulus estimate} Let $R,\lambda>0$. Recall the definition
of $A_{t}(R,\lambda)$ in (\ref{eq def A_t}). Let $(u,\Omega)$ be
a solution of mean curvature flow without singularities in $\R^{n+1}$
($n\ge2$). Suppose $A_{0}(R,\lambda)\Subset\Omega_{0}$ and 
\begin{equation}
u(x,0)\le-\frac{1}{\lambda}\log\cos\big[\lambda(|x|-R)\big]\qquad\text{ for all }x\in A_{0}(R,\lambda)\;.
\end{equation}
Then for $t\in I(R)=[0,\frac{R^{2}}{2(n-1)}]$ there holds $A_{t}(R,\lambda)\Subset\Omega_{t}$
and 
\begin{equation}
u(x,t)\le w(x,t)=-\frac{1}{\lambda}\log\cos\left[\lambda\left(\rad-R\right)\right]+\lambda\,t+f(t)\label{eq thm annulus estimate}
\end{equation}
holds for all $x\in A_{t}(R,\lambda).$ The expression $f(t)$ can
be given by (\ref{eq 1. choice of f}) or (\ref{eq 2. choice of f})
depending on your preferences. 
\end{thm}

\begin{proof}
From \cite{Mau1} we know that $A_{t}(R,\lambda)\Subset\Omega_{t}$
for all $t>0$, because the boundary of $A_{t}(R,\lambda)$ moves
by mean curvature flow. In particular, $\graph u(\cdot,t)$ will not
touch $\graph w(\cdot,t)$ at infinity. As we have demonstrated above,
there will be no touching in the interior, either. Thus, (\ref{eq thm annulus estimate})
holds and the theorem is proven.
\end{proof}
Clearly $-w$ resembles a subsolution, and consequently, $\graph-w(\cdot,t)$
is a lower barrier which cannot be touched from above by any mean
curvature flow.

\subsection{Example with wild oscillations}

We can use Theorem \ref{thm annulus estimate} to construct an example
of mean curvature flow without singularities which behaves quite badly.
The associated domain $\Omega$ simply describes a shrinking ball
and the associated function $u$ converges to $+\infty$ at $\partial\Omega$.
So the graphical surface $M_{t}=\graph u(\cdot,t)$ vanishes at $+\infty$
the moment when $\Omega_{t}$ shrinks to a point. Special about the
surface is that it does not vanish monotonically to infinity. Instead,
$u(0,t)$ increasingly oscillates as $t$ approaches the final time.
The surface $M_{t}$ has unbounded curvature at any time, a behavior
that is not mirrored by the domain $\Omega_{t}$. In fact, as we go
upwards at a fixed time $t$, $M_{t}$ has to get closer and closer
to $\partial\Omega_{t}\times\R$. However, $M_{t}$ is not smoothly
asymptotic to $\partial\Omega_{t}\times\R$ but $M_{t}$ approaches
$\Omega_{t}\times\R$ with more and more sheets. Figure \ref{fig oscillation}
is an attempt to depict $M_{t}$. For simplicity we will only consider
the case $n=2$, although it is straightforward to generalize to higher
$n$.

\begin{figure}
\label{fig oscillation}
\begin{centering}
\includegraphics[height=6cm]{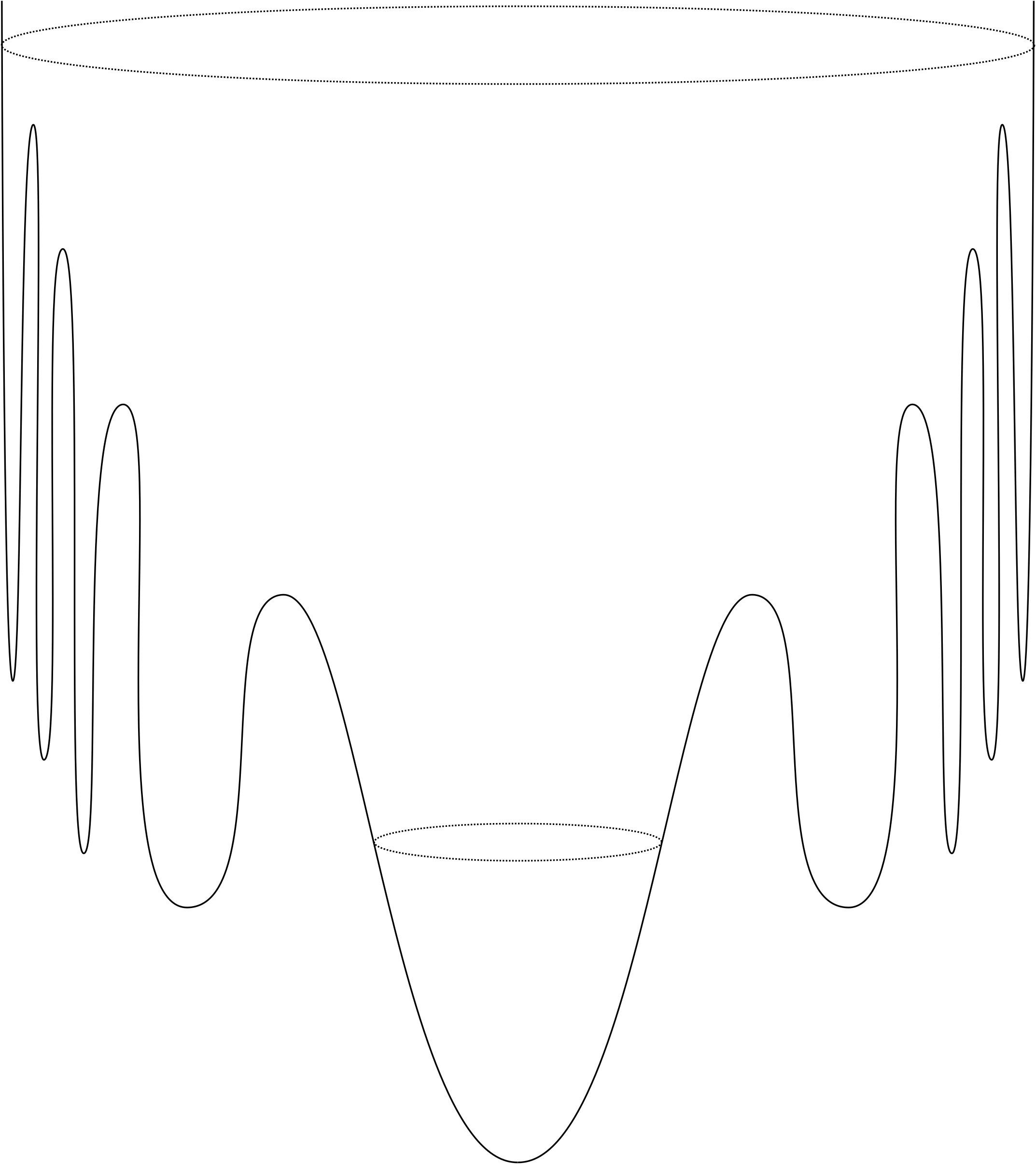}
\par\end{centering}
\caption{Wildly Oscillating Mean Curvature Flow Without Singularities.}
\end{figure}

For all $k\in\mathbb{N}$ with $k\ge2$ we define the Intervals $I_{k}\coloneqq\left(\frac{k-2}{k-1},\frac{k-1}{k}\right)$.
The length of $I_{k}$ is given by $\frac{1}{(k-1)k}$. Let $\lambda_{k}\coloneqq(k-1)k\,\pi$
and let $R_{k}$ be the centers of the $I_{k}$. Then we define 
\begin{align}
w_{+}(x) & \coloneqq\begin{cases}
-\frac{1}{\lambda_{k}}\log\cos\big[\lambda_{k}\,(|x|-R_{k})\big]+k & \text{for }|x|\in I_{k}\text{ with even }k,\\
+\infty & \text{else},
\end{cases}\\
w_{-}(x) & \coloneqq\begin{cases}
+\frac{1}{\lambda_{k}}\log\cos\big[\lambda_{k}\,(|x|-R_{k})\big]+k^{3} & \text{for }|x|\in I_{k}\text{ with odd }k,\\
-\infty & \text{else.}
\end{cases}
\end{align}
The functions $w_{\pm}$ are continuous and $w_{-}<w_{+}$ holds,
in fact, even $w_{+}-w_{-}=\infty$ is true. Figure \ref{fig comb}
sketches $w_{+}$ and $w_{-}$. The different added constants $k$
and $k^{3}$ are there to ensure that the barriers corresponding to
$w_{+}$ and $w_{-}$ stay interlocked for all times such that a mean
curvature flow between those is bound to oscillate infinitely towards
$\partial\Omega_{t}\times\R$. But we are going to have a closer look
at this now. 
\begin{figure}
\begin{centering}
\includegraphics[height=5.5cm]{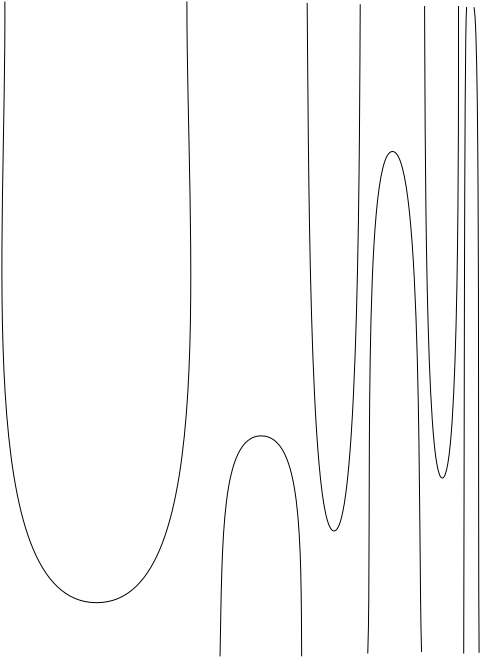}
\par\end{centering}
\caption{Sketch of a cross-section for $w_{\pm}$.}
\label{fig comb}
\end{figure}

Let $u_{0}$ be any smooth function on $B_{1}(0)\subset\R^{2}$ with
$w_{-}<u_{0}<w_{+}$ and with $u_{0}(x)\to\infty$ for $x\to\partial B_{1}(0)$.
%Certainly, there are plenty of those.
Let $u\colon\Omega\to\R$ be a mean curvature flow without singularities
that starts from $u_{0}$. The solution is defined up to the time
$T=\frac{1}{2}$, at which the balls $\Omega_{t}$ shrink to a point.

Let $T_{k}\coloneqq\frac{R_{k}^{2}}{2}$. Theorem \ref{thm annulus estimate}
yields for even $k$ 
\begin{equation}
\begin{split}u(0,T_{k}) & \le k-\frac{1}{\lambda_{k}}\log\cos\left[\lambda_{k}\left(\sqrt{2T_{k}}-R_{k}\right)\right]+\left(\lambda_{k}+\sqrt{\frac{2\lambda_{k}}{\pi\,R_{k}}}\right)T_{k}\\
 & =k+\left((k-1)\,k\,\pi+\sqrt{\frac{2(k-1)\,k}{R_{k}}}\right)T_{k}\\
 & \simeq\frac{\pi}{2}\,k^{2}\qquad(k\to\infty)\;.
\end{split}
\end{equation}
Analogously, we obtain for odd $k$ 
\begin{equation}
\begin{split}u(0,T_{k}) & \ge k^{3}-\left((k-1)\,k\,\pi+\sqrt{\frac{2(k-1)\,k}{R_{k}}}\right)T_{k}\\
 & \simeq k^{3}\qquad(k\to\infty)\;.
\end{split}
\end{equation}
This says that $u(0,T_{2l+1})$ is larger than $\frac{1}{2}(2l+1)^{3}$,
while $u(0,T_{2l})$ is smaller than $2(2l)^{2}$ if the number $l\in\mathbb{N}$
is large enough. This shows that $u(0,t)$ has no clear rate with
which it tends to infinity. Since $T_{k}\to\frac{1}{2}$ it also shows
that $u(0,t)$ oscillates more and more on smaller and smaller time
intervals and that there must be a sequence $(t_{k})_{k\in\mathbb{N}}$
of times with $t_{k}\to T=\frac{1}{2}$ and $\dot{u}(0,t_{k})\to-\infty$,
while it is clear that there are also sequences $(\tilde{t}_{k})_{k\in\mathbb{N}}$
with $\tilde{t}_{k}\to T$ and $\dot{u}(0,\tilde{t}_{k})\to+\infty$.
In summary, one can say that much of the behavior of $u_{0}$ for
$|x|\to R=1$ is transmitted by virtue of Theorem \ref{thm annulus estimate}
to $u(0,t)$ for $t\to T$. In the next paragraph we will see more
of this.
\begin{rem*}
An example of a mean curvature flow without singularities akin to
the one we have discussed is defined on a half-space. It also has
unbounded curvature for all $t\ge0$ and it sheets towards a plane.
It can be constructed in a similar fashion but instead of the annulus
barrier one uses grim reapers cross $\R$ in that case.
\end{rem*}

\subsection{Relation between spatial and temporal asymptotics}

Theorem \ref{thm annulus estimate} shows that the height of a solution
can be estimated by the height of the solution on an annulus at a
prior time. This can be used to derive a relationship between temporal
and spatial asymptotics. 
\begin{thm}
\label{thm relation between asymptotics} Let $n\ge2$ and let $u_{0}\colon\R^{n}\supset B_{\rho}(0)\to\R$
be a smooth function and suppose that there are $\alpha>1$ and $C>0$
such that 
\begin{equation}
u_{0}(x)\simeq C\,(\rho-|x|)^{-\alpha}\qquad(|x|\to\rho)\;.
\end{equation}
Let $u\colon\Omega\to\R$ be a mean curvature flow without singularities
starting from $u_{0}$. Then 
\begin{equation}
\begin{split}u(0,t) & \simeq C\left(\rho-\sqrt{2(n-1)t}\right)^{-\alpha}\qquad(t\to T)\\
 & \simeq C\left(\frac{n-1}{\rho}\,(T-t)\right)^{-\alpha}\qquad(t\to T)
\end{split}
\end{equation}
holds with $T\coloneqq\frac{\rho^{2}}{2(n-1)}$.
\end{thm}

\begin{proof}
The idea is to use the barriers over annuli as depicted in Figure
\ref{fig relation}.

\begin{figure}

\begin{centering}
\includegraphics[width=0.5\textwidth]{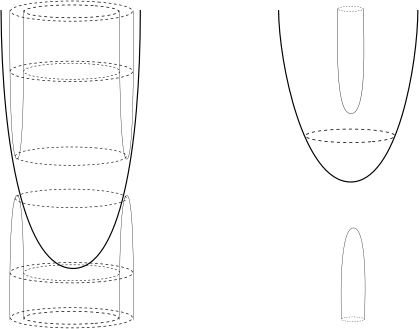}
\par\end{centering}
\caption{Use of the barriers for relation of spatial and temporal asymptotics.}
\label{fig relation}

\end{figure}

Let $0<t<T$, $r_{t}\coloneqq\sqrt{2(n-1)t}$, and $\lambda_{t}\coloneqq\frac{\pi}{2}\,(\rho-r_{t})^{-\frac{1+\alpha}{2}}$.
We set 
\begin{equation}
A(t)\coloneqq A_{0}(r_{t},\lambda_{t})=\left\{ x\colon r_{t}-\frac{\pi}{2\lambda_{t}}<|x|<r_{t}+\frac{\pi}{2\lambda_{t}}\right\} .
\end{equation}
We will assume that $t$ is sufficiently close to $T$ such that $\rho-r_{t}>(\rho-r_{t})^{\frac{1+\alpha}{2}}$
and hence $A(t)\Subset B_{\rho}(0)$ hold.

Theorem \ref{thm annulus estimate} yields 
\begin{equation}
u(0,t)\le\sup_{x\in A(t)}u_{0}(x)+\left(\lambda_{t}+\sqrt{\frac{2\lambda_{t}}{\pi\,r_{t}}}\right)t\;.\label{eq thm rba leq}
\end{equation}
We determine the asymptotic behavior: Let $x_{t}\in\overline{A(t)}$
with $\sup_{x\in A(t)}u_{0}(x)=u_{0}(x_{t})$. We observe that 
\begin{equation}
\frac{\rho-r_{t}\pm\frac{\pi}{2\lambda_{t}}}{\rho-r_{t}}=\frac{\rho-r_{t}\pm(\rho-r_{t})^{(1+\alpha)/2}}{\rho-r_{t}}=1\pm(\rho-r_{t})^{\frac{\alpha-1}{2}}\simeq1\qquad(t\to T)\;.
\end{equation}
From this we infer that $\rho-|x_{t}|\simeq\rho-r_{t}$ holds and
we deduce 
\begin{equation}
\sup_{x\in A(t)}u_{0}(x)=u_{0}(x_{t})\simeq C\,(\rho-|x_{t}|)^{-\alpha}\simeq C\,(\rho-r_{t})^{-\alpha}\qquad(t\to T)\;.\label{eq thm rba sup}
\end{equation}

Similar to (\ref{eq thm rba leq}) and (\ref{eq thm rba sup}), we
find 
\begin{equation}
u(0,t)\ge\inf_{x\in A(t)}u_{0}(x)-\left(\lambda_{t}+\sqrt{\frac{2\lambda_{t}}{\pi\,r_{t}}}\right)t\label{eq thm rba geq}
\end{equation}
and 
\begin{equation}
\inf_{x\in A(t)}u_{0}(x)\simeq C\,(\rho-r_{t})^{-\alpha}\qquad(t\to T)\;.\label{eq thm rba inf}
\end{equation}
Lastly, we note that 
\begin{equation}
\left(\lambda_{t}+\sqrt{\frac{2\lambda_{t}}{\pi\,r_{t}}}\right)t\simeq\frac{\pi}{2}\,(\rho-r_{t})^{-(1+\alpha)/2}\,T\qquad(t\to T)\;.\label{eq thm rba rest}
\end{equation}
Taking together (\ref{eq thm rba leq}), (\ref{eq thm rba sup}),
(\ref{eq thm rba geq}), (\ref{eq thm rba inf}), and (\ref{eq thm rba rest}),
we find (note $\alpha>\frac{1+\alpha}{2}$) 
\begin{equation}
u(0,t)\simeq C\,(\rho-r_{t})^{-\alpha}=C\left(\rho-\sqrt{2(n-1)t}\right)^{-\alpha}\qquad(t\to T)\;.
\end{equation}
With a Taylor expansion at $t=T=\frac{\rho^{2}}{2(n-1)}$, we can
also write this as 
\begin{equation}
u(0,t)\simeq C\left(\frac{n-1}{\rho}(T-t)\right)^{-\alpha}\qquad(t\to T)\;.\qedhere
\end{equation}
\end{proof}
\begin{rem*}
The assumption on the rate of $u_{0}$ is not expected to be optimal:
In \cite{IW} and \cite{IWZ} rotational symmetric solutions of the
mean curvature flow without singularities that fit into our setting
have been constructed. Their solutions have the asymptotics Theorem
\ref{thm relation between asymptotics} is concerned with, while Isenberg
and Wu only state the asymptotics for blow-ups. Their results concentrate
more on the blow-up rate of the curvature and they prove that any
blow-up rate $(T-t)^{-\alpha}$ with $\alpha\ge1$ for the curvature
is possible. The case $\alpha=1$ is dealt with in \cite{IWZ} and
is beyond the scope of our Theorem \ref{thm relation between asymptotics}.
\end{rem*}

\appendix

\section{Normal graphs}

\label{sec normal-graphs}
\begin{defn}[Tubular Neighborhood]
\label{def tubular neighborhood} Let $N$ be a $C^{2}$-hypersurface
of $\R^{n+1}$ and let $\nu$ be a continuous normal to $N$. If for
$\delta>0$ ($\delta=\infty$ is possible) the map $\Phi\colon(p,d)\mapsto p-d\,\nu$,
defined on $N\times(-\delta,\delta)$, is a diffeomorphism onto its
image, then this image is called a \emph{tubular neighborhood of $N$
of thickness $\delta$} and we denote it by $N_{\delta}$. The diffeomorphism
$\Phi$ is called \emph{tubular diffeomorphism}.
\end{defn}

\begin{defn}[Normal Graph]
\label{def normal graph} Let $N$ be a $C^{2}$-hypersurface of
$\R^{n+1}$ with tubular diffeomorphism $\Phi\colon N\times(-\delta,\delta)\to N_{\delta}$.
A hypersurface $M$ of $\R^{n+1}$ is called a \emph{normal graph
over $N$} (or \emph{has a normal graph representation over $N$})
if $M=\Phi(\graph u)$ for some function $u\colon N\supset N'\to(-\delta,\delta)$.
This function is called the \emph{representation function}.
\end{defn}

\subsection{Geometry of normal graphs}
\begin{prop}
\label{prop geometry of normal graphs} Let $M$ be a normal graph
over a $C^{3}$-hypersurface $N^{n}$ of $\R^{n+1}$ with representation
function $u$. We denote the metric and second fundamental form of
$N$ by $g_{ij}$ and $h_{ij}$ and the Weingarten map with $A$.
On $M$ we choose the normal $\nu^{M}$ which satisfies $\left\langle \nu^{M},\nu\right\rangle \ge0$,
where $\nu$ is the normal of $N$. Then, with $B\coloneqq\sum_{a=0}^{\infty}(u\,A)^{a}$,
\begin{align}
g_{ij}^{M} & =g_{ij}+u_{i}\,u_{j}-2\,u\,h_{ij}+u^{2}\,h_{ik}\,h_{j}^{k}\;,\label{eq gong g_ij}\\
h_{ij}^{M} & =\frac{1}{\sqrt{1+\left|\mathrm{D}u\cdot B\right|^{2}}}\big[h_{ij}+u_{ij}-u\,h_{j}^{k}\,h_{kj}+\big(u_{i}\,h_{j}^{k}+u_{j}\,h_{i}^{k}+u\,\nabla_{j}h_{i}^{k}\big)\,u_{l}\,B_{k}^{l}\big]\;.\label{eq gong h_ij}
\end{align}
\end{prop}

\begin{proof}
We will denote the embedding of $N$ into $\R^{n+1}$ with $N$, too.
A parametrization of $M$ is given by $X(x)=N(x)-u(x)\,\nu(x)$ with
$x\in N'\subset N$. We compute derivatives up to the second order
while noting $N_{ij}=-h_{ij}\,\nu$ and $\nu_{i}=h_{i}^{k}\,N_{k}$:
\begin{align}
X_{i} & =N_{i}-u_{i}\,\nu-u\,\nu_{i}=N_{i}-u_{i}\,\nu-u\,h_{i}^{k}\,N_{k}\;,\\
X_{ij} & =N_{ij}-u_{ij}\,\nu-u_{i}\,\nu_{j}-u_{j}\,h_{i}^{k}\,N_{k}-u\,(\nabla_{j}h_{i}^{k})\,N_{k}-u\,h_{i}^{k}\,N_{kj}\nonumber \\
 & =(-h_{ij}-u_{ij}+u\,h_{i}^{k}\,h_{kj})\,\nu-(u_{i}\,h_{j}^{k}+u_{j}\,h_{i}^{k}+u\,\nabla_{j}h_{i}^{k})\,N_{k}\;.
\end{align}

The asserted identity (\ref{eq gong g_ij}) for $g_{ij}^{M}=\left\langle X_{i},X_{j}\right\rangle $
follows.

For the second identity we compute $h_{ij}^{M}=-\left\langle X_{ij},\nu^{M}\right\rangle .$
\global\long\def\tu{\tilde{u}}%
 To this end, we observe that the submanifold $M$ is given by $d-\tu=0$,
where $d$ is the distance function to $N$ and $\tu=u\circ\pi$ is
the extension of $u$ to the tubular neighborhood of $N$ which is
constant in normal direction; $\pi$ is the projection to the closest
point in $N$. (We actually have $\Phi^{-1}=(\pi,d)$, where $\Phi$
is the tubular diffeomorphism.) So for $v\in\mathrm{T}_{p}M$ we have
$(\mathrm{D}d-\mathrm{D}\tu)|_{p}\cdot v=0$. From here we see that
$(\nabla d-\nabla\tu)\circ X$ is proportional to the normal $\nu^{M}$
of $M$. It actually points in the direction of $-\nu^{M}$. Moreover,
$|\nabla d-\nabla\tu|^{2}=1+|\nabla\tu|^{2}$ since $\nabla d$ and
$\nabla\tu$ are orthogonal and $|\nabla d|=1$. Using $\nabla d=-\nu\circ\pi$
and $\mathrm{D}\tu|_{X}\cdot\nu=0$, we obtain 
\begin{equation}
\begin{split}h_{ij}^{M} & =-\left\langle X_{ij},\nu^{M}\right\rangle =\left(\frac{\mathrm{D}d-\mathrm{D}\tu}{\sqrt{1+|\mathrm{D}\tu|^{2}}}\circ X\right)\cdot X_{ij}\\
 & =\frac{1}{\sqrt{1+\left|\nabla\tu|_{X}\right|^{2}}}\left[h_{ij}+u_{ij}-u\,h_{i}^{k}\,h_{kj}+\left(u_{i}\,h_{j}^{k}+u_{j}\,h_{i}^{k}+u\,\nabla_{j}h_{i}^{k}\right)\,\mathrm{D}\tu|_{X}\cdot N_{k}\right]\;.
\end{split}
\label{eq gong h_ij 1}
\end{equation}
Now we turn our attention to $\mathrm{D}\tu|_{X}=(\mathrm{D}u|_{\pi}\cdot\mathrm{D}\pi)|_{X}=\mathrm{D}u\cdot\mathrm{D}\pi|_{X}$
and focus on $\mathrm{D}\pi|_{X}$. Because $\textrm{id}=\Phi(\pi,d)=N(\pi)-d\,\nu(\pi)$
holds, we compute, using $\mathrm{D}\nu=\mathrm{D}N\cdot A$ and $\nabla d|_{X}=-\nu$,
\begin{equation}
\begin{split}\mathrm{id}_{\mathrm{T}N_{\delta}}|_{X} & =\mathrm{D}N\cdot\mathrm{D}\pi|_{X}-d|_{X}\,\mathrm{D}\nu\cdot\mathrm{D}\pi|_{X}-\nu\otimes\mathrm{D}d|_{X}\\
 & =\mathrm{D}N\cdot\left(\mathrm{id}_{\mathrm{T}N}-d|_{X}\,A\right)\cdot\mathrm{D}\pi|_{X}+\nu\otimes\nu^{\flat}\;.
\end{split}
\end{equation}
We observe that $d|_{X}=u$ holds. Because $\nu^{\flat}\cdot\nabla N=\left\langle \nu,\nabla N\right\rangle =0$,
a multiplication of $\nabla N$ from the right and a subsequent cancellation
of $\nabla N$ on the left yields 
\begin{equation}
\mathrm{id}_{\mathrm{T}N}=\left(\mathrm{id}_{\mathrm{T}N}-u\,A\right)\cdot\mathrm{D}\pi|_{X}\cdot\mathrm{D}N\,.
\end{equation}
We only mention here that $\|u\,A\|<1$ holds because the thickness
of the maximal tubular neighborhood is bounded by the curvature in
that way, though there might be global effects further cutting down
the maximal thickness. By the Neumann series, $(\mathrm{id}_{\mathrm{T}N}-u\,A)^{-1}=\sum_{a=0}^{\infty}(u\,A)^{a}\eqqcolon B$.
Hence, we obtain the identity 
\begin{equation}
\mathrm{D}\pi|_{X}\cdot\mathrm{D}N=B\;.
\end{equation}
From this equation we deduce 
\begin{equation}
\mathrm{D}\tu|_{X}\cdot\mathrm{D}N=\mathrm{D}u\cdot B\;.\label{eq gong tu}
\end{equation}
Because the metric on $N$ is the metric on $\R^{n+1}$ pulled back
via the embedding $N$, we have 
\begin{equation}
|\mathrm{D}\tu_{|X}|_{\R^{n+1}}=|\mathrm{D}\tu_{|X}\cdot\mathrm{D}N|_{g}=|\mathrm{D}u\cdot B|_{g}\;.\label{eq gong norm}
\end{equation}
Substituting (\ref{eq gong tu}) and (\ref{eq gong norm}) into (\ref{eq gong h_ij 1})
yields (\ref{eq gong h_ij}).
\end{proof}

\subsection{Normal graphs and local graph representations}
\begin{defn}
Let $M$ be a Riemannian manifold. We say a point $x\in M$ \emph{has
buffer $r>0$ in $M$} if any curve $\gamma\colon[0,a)\to M$ of length
at most $r$ and with $\gamma(0)=x$ is extendable to a curve $\gamma\colon[0,r]\to M$.

A subset \emph{$M'\subset M$ has buffer $r$ in $M$} if every point
of $M'$ has buffer $r$ in $M$.
\end{defn}

\begin{defn}
\label{def admit graph of radius} Let $M$ be a hypersurface. We
say \emph{$M$ admits local graph representations of radius $r>0$}
if $B_{r}(x)\cap M$ is graphical over the tangential hyperplane at
$x$ for any point $x\in M$.
\end{defn}

The following lemma says that normal graphs over a given hypersurface
admit local graph representations of a controlled radius.
\begin{prop}
\label{prop ntg} Let $N$ be a complete hypersurface and let $\delta$
be the thickness of a tubular neighborhood of $N$. Let $M$ be a
normal graph over $N$ with representation function $u\colon N\to(-\frac{\delta}{2},\frac{\delta}{2})$.
We assume that $|\mathrm{D}u|\le L<\frac{1}{6}$ holds. Let $N'\subset N$
have buffer $\rho>0$ and suppose that the distance functions $\dist_{\R^{n}}$
and $\dist_{N}$ are equivalent on $N$: $\dist_{N}\le C_{0}\,\dist_{\R^{n}}$.

Then there is $r>0$ such that $M$ admits local graph representations
of radius $r$. The radius $r$ depends only on $L$, $\rho$, $C_{0}$,
and $\sup|A_{N}|$. Moreover, the gradients in the local graph representations
are bounded by $8L$.
\end{prop}

\begin{proof}
From the proof of Proposition \ref{prop geometry of normal graphs}
we know that 
\begin{equation}
\left\langle \nabla d,-\nu_{M}\right\rangle =\frac{1}{\sqrt{1+|\mathrm{D}u\cdot B|^{2}}}\label{eq ntg alt davor}
\end{equation}
holds with $B=(I-uA)^{-1}=\sum_{k=0}^{\infty}(uA)^{k}$. Because of
$|u|<\frac{\delta}{2}$ the eigenvalues of $uA$ are at most $\frac{1}{2}$
in modulus. So the eigenvalues of $B$ are bounded by $2$. Thus,
from the gradient estimate $|\mathrm{D}u|\le L$ we obtain 
\begin{equation}
\left\langle \nabla d,-\nu_{M}\right\rangle \ge\frac{1}{\sqrt{1+(2|\mathrm{D}u|)^{2}}}\ge\frac{1}{\sqrt{1+(2L)^{2}}}\;,\label{eq ntg alt}
\end{equation}
where $d$ is the distance function to $N$.

Let $x_{0},x\in M$. Then 
\begin{equation}
\left\langle \nabla d(x_{0}),-\nu_{M}(x)\right\rangle =\left\langle \nabla d(x),-\nu_{M}(x)\right\rangle +\left\langle \nabla d(x_{0})-\nabla d(x),-\nu_{M}(x)\right\rangle \;.\label{eq ntg 0-trick}
\end{equation}
For a curve $\gamma\colon[0,1]\to M$ with $\gamma(0)=x_{0}$ and
$\gamma(1)=x$, we have 
\begin{equation}
\nabla d(x)-\nabla d(x_{0})=\int_{0}^{1}\frac{\mathrm{d}}{\mathrm{d}t}\nabla d(\gamma)\,\mathrm{d}t=\int_{0}^{1}\nabla^{2}d(\gamma)\cdot\gamma'\,\mathrm{d}t\;,
\end{equation}
and therefore 
\begin{equation}
|\nabla d(x_{0})-\nabla d(x)|\le\|\nabla^{2}d\|_{L^{\infty}(N_{{\delta}/2})}\,\dist_{M}(x,x_{0})\le C\,\dist_{\R^{n}}(x,x_{0})\;.\label{eq ntg Lip}
\end{equation}
If $\dist_{\R^{n}}(x,x_{0})\le r$ for sufficiently small $r>0$,
we obtain from (\ref{eq ntg 0-trick}), (\ref{eq ntg alt}), and (\ref{eq ntg Lip})
\begin{equation}
\left\langle \nabla d(x_{0}),-\nu_{M}(x)\right\rangle \ge\frac{1}{\sqrt{1+(2L)^{2}}}-C\,\dist_{\R^{n}}(x,x_{0})\ge\frac{1}{\sqrt{1+(3L)^{2}}}\;.
\end{equation}
That means that in the ball $B_{r}(x_{0})$, $M$ is locally graphical
over the hyperplane determined by $\nabla d(x_{0})$. The gradient
is bounded by $3L$ because for a hyperplane the corresponding endomorphism
$B$ is the identity. Thus, $M$ is also locally graphical over the
hyperplane $\textrm{T}_{x_{0}}M$ with gradient bounded by $8L$.
The calculation goes like this: Let $\xi,\eta\in\R^{n+1}$ be two
unit-vectors with $\xi^{1},\eta^{1}\ge(1+(3L)^{2})^{-1/2}$. Then
the remaining components $\hat{\xi}=(\xi^{2},\ldots,\xi^{n+1})$ satisfy
\[
\left|\hat{\xi}\right|^{2}=1-(\xi^{1})^{2}\le1-\frac{1}{1+(3L)^{2}}=\frac{(3L)^{2}}{1+(3L)^{2}}\;,
\]
and analogous for $\eta$. Hence, 
\[
\left\langle \xi,\eta\right\rangle =\xi^{1}\,\eta^{1}+\left\langle \hat{\xi},\hat{\eta}\right\rangle \ge\frac{1}{1+(3L)^{2}}-\frac{(3L)^{2}}{1+(3L)^{2}}=\frac{1-(3L)^{2}}{1+(3L)^{2}}\;.
\]
We want to write this in the form (compare to (\ref{eq ntg alt davor}))
\[
\left\langle \xi,\eta\right\rangle \ge\frac{1-(3L)^{2}}{1+(3L)^{2}}\stackrel{!}{=}\frac{1}{\sqrt{1+L'^{2}}}.
\]
We obtain 
\[
L'=\sqrt{\frac{(1+(3L)^{2})^{2}}{(1-(3L)^{2})^{2}}-1}=\frac{\sqrt{4(3L)^{2}}}{1-(3L)^{2}}\stackrel{L<\frac{1}{6}}{<}8L\;.
\]

Now, we need to prove that not only the normals point in the right
directions and we have local (on $M$) graph representations, but
we also need to prove that the projection to the hyperplane $\mathrm{T}_{x_{0}}M$
is one-to-one. Assume that this is not the case and there exist two
points whose projections on $\mathrm{T}_{x_{0}}M$ are the same. Let
$a>0$ be the Euclidean distance of those points. Because $M$ is
a normal graph over $N$ with gradient bounded by $\frac{1}{6}$ and
$M$ is in a tubular neighborhood of half the maximal thickness (we
are sufficently far away from focal points), we can make the radius
$r$ small depending on $\sup|A_{N}|$, that in this case the projection
of the two points to $N$ have $N$-distance bounded by $\frac{1}{2}a$.
Because $M$ is a normal graph over $N$ with gradient bound $\frac{1}{6}$,
the difference in distance to $N$ of the two points must be bounded
by $\frac{1}{12}a$. But if $r$ is sufficiently small compared $\sup|A_{N}|$,
then the Euclidean squared distance $|p-q|^{2}$ of two points is
comparable to $|d(p)-d(q)|^{2}+|\pi_{N}(p)-\pi_{N}(q)|^{2}$. As we
can see, this is not the case for the above points. Therefore, the
assumption that there exist two points in $M\cap B_{r}(x_{0})$ whose
projections to $\mathrm{T}_{x_{0}}M$ agree is false. We conclude
that $M\cap B_{r}(x_{0})$ is graphical over $\mathrm{T}_{x_{0}}M$.
\end{proof}

\subsection{Hypersurfaces close to each other}

In this paragraph, we show that a hypersurface $M$ with bounded curvature
which lies sufficiently close to a hypersurface $N$ can be written
as a normal graph over $N$. In fact, the $C^{1}$-norm of the graphical
representation is small if $M$ is only close enough to $N$; very
much in the spirit of the interpolation inequality $\|u\|_{C^{1}}^{2}\le C\,\|u\|_{C^{0}}\,\|u\|_{C^{2}}$.
\begin{prop}
\label{prop close hypersurfaces} Let $N\subset\R^{n+1}$ be a hypersurface
with bounded curvature $|A_{N}|\le C<\infty$. Let $0<\delta<(2\,C)^{-1}$
be chosen such that $\delta$ is the thickness of a tubular neighborhood
of $N$, denoted by $N_{\delta}$. Let $M\subset N_{\delta}$ be a
hypersurface, also of bounded curvature $|A_{M}|\le C$. Let $x_{0}\in M$
be a point with buffer $r>0$ in $M$. Then holds 
\begin{equation}
|\nabla^{M}d|^{2}(x_{0})\le\delta\cdot\max\left\{ 6\,C,\,\left(\frac{\pi}{r}\right)^{2}\delta\right\} \;.
\end{equation}
Herein, $d\colon N_{\delta}\to(-\delta,\delta)$ denotes the signed
distance to $N$.
\end{prop}

\begin{proof}
Without loss of generality, we assume $d(x_{0})\ge0$. Locally around
$x_{0}$, we choose a continuous normal $\nu_{M}$ of $M$ such that
$w_{0}\coloneqq w(x_{0})\coloneqq\left\langle \nabla d,-\nu_{M}\right\rangle |_{x_{0}}\ge0$.
A continuous choice of a normal is always possible in $M$-balls with
radius bounded by $\frac{\pi}{2C}$. This is because we can join any
two points of the ball by a curve through $x_{0}$ of length less
than $\frac{\pi}{C}$. Noting that $C$ bounds the curvature, the
direction of a normal is changing along such a curve by an angle less
than $\pi$. But any point admits only two directions for a normal,
which are separated by an angle of $\pi$. Thus, a normal field which
is continuously extended along ``short'' curves starting from $x_{0}$
is continuous.

In what follows, we note that $|\nabla d|=1$, $\nabla^{2}d(\cdot,\nabla d)=0$
and $|\nabla^{2}d|_{\mathrm{op}}\le2\,C\,$ hold, where $|\cdot|_{\mathrm{op}}$
denotes the operator norm (largest eigenvalue in modulus). This inequality
follows from $|d|<\delta<(2\,C)^{-1}$ and that the eigenvalues of
$\nabla^{2}d$ which correspond to directions perpendicular to $\nabla d$
are given by $-\frac{\kappa_{i}}{1-d\,\kappa_{i}}$, where $\kappa_{i}$
denote the principal curvatures of $N$ at the closest point on $N$.

The assertion is trivial if $|\nabla^{M}d(x_{0})|=0$ holds. Therefore,
we assume $|\nabla^{M}d(x_{0})|\neq0$ from now on. Let $x(t)$ be
the maximal solution of the initial value problem 
\begin{equation}
\dot{x}(t)=\frac{\nabla^{M}d(x(t))}{|\nabla^{M}d|(x(t))}=\frac{\nabla d(x(t))+w(x(t))\,\nu_{M}(x(t))}{\sqrt{1-w^{2}(x(t))}}
\end{equation}
with $x(0)=x_{0}$ and $w\coloneqq\left\langle \nabla d,-\nu_{M}\right\rangle $.
With the second fundamental form $h$ of $M$ and noting that $|\nabla^{M}d|=\sqrt{1-w^{2}}=|\nu_{M}+w\,\nabla d|$,
along this curve holds 
\begin{equation}
\begin{split}\left|\frac{\mathrm{d}}{\mathrm{d}t}w(x(t))\right| & =\left|\left\langle \nabla_{\dot{x}}\nabla d,-\nu_{M}\right\rangle +\left\langle \nabla d,-\nabla_{\dot{x}}\nu_{M}\right\rangle \right|\\
 & =\left|-\nabla^{2}d(\dot{x},\nu_{M}+w\,\nabla d)-h(\dot{x},\nabla^{M}d)\right|\\
 & \le3\,C\sqrt{1-w^{2}}\;.
\end{split}
\end{equation}
It follows for the change in angle $\left|\frac{\mathrm{d}}{\mathrm{d}t}\arccos(w(x(t)))\right|\le3\,C$,
and therefore 
\begin{equation}
\arccos(w(x_{0}))-3\,C\,t\le\arccos(w(x(t))\le\arccos(w(x_{0}))+3\,C\,t\;.\label{eq ceo arccos}
\end{equation}
Thus, the solution $x(t)$ exists for times $0\le t<\min\{r,T\}$,
where $T\coloneqq\frac{\arccos(w(x_{0}))}{3\,C}$. It is important
here that $\alpha\coloneqq\arccos(w(x_{0}))\le\frac{\pi}{2}$ which
follows from the choice of the normal such that $w_{0}\coloneqq w(x_{0})\ge0$.
This is also crucial for the inequality in the third line of the following
estimate. 
\begin{equation}
\begin{split}\delta & \ge d(x(t))-d(x(0))=\int_{0}^{t}\frac{\mathrm{d}}{\mathrm{d}t}d(x(s))\,\mathrm{d}s=\int_{0}^{t}\left\langle \nabla d,\dot{x}\right\rangle \,\mathrm{d}s\\
 & =\int_{0}^{t}\sqrt{1-w^{2}(x)}\,\textrm{d}s=\int_{0}^{t}\sin\arccos w(x)\,\textrm{d}s\\
 & \ge\int_{0}^{t}\sin(\arccos(w_{0})-3\,C\,s)\,\textrm{d}s=\frac{1}{3\,C}(\cos(\alpha-3\,C\,t)-w_{0})\\
 & =\frac{1}{3\,C}
\end{split}
\end{equation}
From this we infer 
\begin{equation}
|\nabla^{M}d|^{2}(x_{0})=1-w_{0}^{2}\le(1+w_{0})\,\frac{\alpha}{t}\,\delta\le2\arccos(w_{0})\,\frac{1}{t}\,\delta\;.
\end{equation}
In the case that the solution $x(t)$ exists up to time $T$, we obtain
$|\nabla^{M}d|^{2}(x_{0})\le6\,C\,\delta$ from the last inequality
by $t\to T$. Otherwise, from $t\to r$ follows (note $\arcsin x\le\frac{\pi}{2}\,x$)
\begin{equation}
\frac{2}{\pi}\,|\nabla^{M}d|(x_{0})\le\frac{|\nabla^{M}d|^{2}(x_{0})}{\arcsin(|\nabla^{M}d|(x_{0}))}=\frac{|\nabla^{M}d|^{2}(x_{0})}{\arccos(w_{0})}\le2\,\frac{1}{r}\,\delta\;.
\end{equation}
The assertion follows.
\end{proof}
\begin{cor}
\label{cor close hypersurfaces} In the situation of Proposition \ref{prop close hypersurfaces},
but with $\delta<(24\,C)^{-1}$ and $\delta<\frac{r}{2\pi}$, we denote
the projection from $M$ onto the closest point in $N$ with $p\colon M\to N$.
Let $M'$ be a subset of $M$ with buffer $r>0$. Then $p|_{M'}$
is a local diffeomorphism and $M'$ is locally graphical over $N$
with gradient bounded by $\max\left\{ \sqrt{12\,C\,\delta},\,\sqrt{2}\frac{\pi}{r}\,\delta\right\} $.
\end{cor}

\begin{proof}
From Proposition \ref{prop close hypersurfaces} we have 
\begin{equation}
|\nabla^{M}d|^{2}\le\delta\,\max\left\{ 6\,C,\,\left(\frac{\pi}{r}\right)^{2}\,\delta\right\} <\frac{1}{4}\;,
\end{equation}
and it follows that in points of $M'$, $\mathrm{D}p$ is an isomorphism
of the corresponding tangential spaces. By the inverse function theorem,
$p$ is a local diffeomorphism. The representation function of the
local graphical representation of $M'$ over $N$ as a normal graph
is denoted by $u$. For the gradient bound we make use of (notice
(\ref{eq ntg alt davor})): 
\begin{equation}
|\nabla^{M}d|^{2}=\frac{|\nabla u\cdot(I-uA)^{-1}|^{2}}{1+|\nabla u\cdot(I-uA)^{-1}|^{2}}\ge\frac{(\frac{6}{7})^{2}|\nabla u|^{2}}{1+(\frac{6}{7})^{2}|\nabla u|^{2}}\,.
\end{equation}
This yields 
\begin{equation}
\begin{split}|\nabla u|^{2} & \le\left(\frac{7}{6}\right)^{2}\,\frac{|\nabla^{M}d|^{2}}{1-|\nabla^{M}d|^{2}}\le\left(\frac{7}{6}\right)^{2}\,\frac{4}{3}\,|\nabla^{M}d|^{2}\le2\,|\nabla^{M}d|^{2}\\
 & \le\max\left\{ 12\,C\,\delta,2\left(\frac{\pi}{r}\right)^{2}\,\delta^{2}\right\} \:.\qedhere
\end{split}
\end{equation}
\end{proof}
In the special situation we face in our applications, we can overcome
that the representation as a normal graph is only local. We formulate
this as 
\global\long\def\hx{\hat{x}}%

\begin{lem}
\label{lem close hypersurfaces} In the situation of Corollary \ref{cor close hypersurfaces},
let now be $M$ specifically given as $M=\graph u$ for a smooth function
$u\colon\Omega\to\R$. The set $\Omega\subset\R^{n}$ is assumed to
be open, smooth, and bounded and we assume $u(\hx)\to\infty$ for
$\hx\to\hx_{0}\in\partial\Omega$. Moreover, let $N$ be specifically
given as $N=\partial\Omega\times\R$.

For $\delta>0$, chosen like in Corollary \ref{cor close hypersurfaces},
there exists $a>0$ such that $M\cap\{x^{n+1}>a\}$ lies in the tubular
neighborhood $N_{\delta}$ and $M\cap\{x^{n+1}>a+1\}$ is a normal
graph over $N$ with gradient bounded by $\max\left\{ \sqrt{12\,C\,\delta},\,\sqrt{2}\,\pi\,\delta\right\} $.
\end{lem}

\begin{proof}
We set $a=\min\{u(x)\colon x\in\Omega,\,\dist(x,\partial\Omega)\ge\delta\}$.
Then $M\cap\{x^{n+1}>a\}$ lies in the tubular neighborhood $N_{\delta}$.

We rename the hypersurfaces such that $M\cap\{x^{n+1}>a\}$ becomes
$M$ and $M\cap\{x^{n+1}>a+1\}$ becomes $M'$. On $M$ we consider
the projection $p$ onto $N$, which by Corollary \ref{cor close hypersurfaces}
is a local diffeomorphism on $M'$. Let $X\coloneqq(\mathrm{D}p)^{-1}(e_{n+1})$
and let $\alpha\colon M'\times[0,\infty)\to M'$ be the flow of the
vector field $X$. Then $p\circ\alpha(x,s)=p(x)+s\,e_{n+1}$ holds
because $\partial_{s}(p\circ\alpha)=\mathrm{D}p\cdot\partial_{s}\alpha=\mathrm{D}p\cdot\mathrm{D}p^{-1}(e_{n+1})=e_{n+1}$
and $\alpha(x,0)=x$.

Let $x=(\hx,x^{n+1})\in M'$ be arbitrary. We consider $x(s)\coloneqq\alpha(x,s)$.
Without loss of generality, we may assume that $e_{1}$ is the normal
vector to $N$ at $p(x)$ and that $p(x)=(0,\ldots,0,h)$ holds. Then
$x(s)$ is of the form 
\begin{equation}
x(s)=\big(x^{1}(s),0,\ldots,0,h+s\big)=\big(x^{1}(s),0,\ldots,0,u(\hx(s))\big)\;.
\end{equation}
It follows $x^{1}(s)\to0$ for $s\to\infty$. Suppose there is another
point $y\in M'$ with $p(y)=p(x)$. Then $y=(y^{1},0,\ldots,0,h)$
holds, and we may assume without loss of generality that $y^{1}$
is between $x^{1}=x^{1}(0)$ and $0$. Because $x^{1}(s)\to0$, there
is $s'>0$ such that $x^{1}(s')=y^{1}$. But then $h=u(\hat{y})=u(\hx)=h+s'$,
a contradiction.

The argument shows that $p$ is injective on $M'$. So $p$ is an
injective local diffeomorphism. As a consequence, $M'$ is a normal
graph over $N$.

The gradient bound follows from Corollary \ref{cor close hypersurfaces}.
\end{proof}

\section{Set flow, domain flow, and $\alpha$-noncollapsed mean curvature
flow}

\label{sec noncollapsedMCF} We follow \cite{HK}, which is based
on \cite{Ilm1,Ilm2,ES,CGG}. We also refer to \cite{Mau1}.
\begin{defn}[Set flow]
 Let $I\subset\R$ be an interval and let $(K_{t})_{t\in I}$ be
a family of closed subsets of $\R^{n+1}$. We say that \emph{$(K_{t})_{t\in I}$
is a set flow} if for any smooth mean curvature flow $(M_{t})_{t\in[t_{0},t_{1}]}$
of closed hypersurfaces and with $[t_{0},t_{1}]\subset I$, we have
\begin{equation}
K_{t_{0}}\cap M_{t_{0}}=\emptyset\quad\Longrightarrow\quad\forall t\in[t_{0},t_{1}]:K_{t}\cap M_{t}=\emptyset\;.
\end{equation}
\end{defn}

\begin{defn}[Level-set flow]
 The level-set flow is the maximal set flow, where ``maximal''
is understood with respect to inclusion of sets. 
\end{defn}

\begin{prop}
For any compact subset $K_{0}\subset\R^{n+1}$ there exists a unique
level-set flow $(K_{t})_{t\in[0,\infty)}$. It coincides with the
level-set flow from the definition of Evans-Spruck and Chen-Giga-Goto.
\end{prop}

\begin{defn}[Domain flow]
\label{def ncMCF domain flow} Let $I\subset\R$ be an interval and
let $(\Omega_{t})_{t\in I}$ be a family of open subsets of $\R^{n+1}$.
Then $(\Omega_{t})_{t\in I}$ is a \emph{domain flow} if for any family
$(K_{t})_{t\in[a,b]}$ of compact subsets of $\R^{n}$ whose boundaries
$(\partial K_{t})_{t\in[a,b]}$ form a classical mean curvature flow,
we have $K_{a}\subset\Omega_{a}\Longrightarrow\forall t\in[a,b]\colon K_{t}\subset\Omega_{t}$;
and the same holds with $\overline{\Omega}^{c}$ instead of $\Omega$.
\end{defn}

\begin{rem}
It can be shown with the help of the Jordan-Brouwer separation theorem
that for any domain flow $(\Omega_{t})_{t\in I}$, $(\overline{\Omega_{t}})_{t\ge0}$
and $(\partial\Omega_{t})_{t\ge0}$ are set flows. However, we cannot
have sudden vanishing, a problem which set flows may exhibit.
\end{rem}

\begin{prop}
Let $K\subset\R^{n+1}$ have a mean convex $C^{2}$-boundary $\partial K$.
Then the level-set flow $(K_{t})_{t\in[0,\infty)}$ starting from
$K_{0}=K$ satisfies $K_{t_{1}}\supset K_{t_{2}}$ for any $t_{1}\le t_{2}$.
\end{prop}

\begin{prop}
\label{prop ncM weak avoidance} If $(K_{t})_{t\in[0,\infty)}$ and
$(L_{t})_{t\in[0,\infty)}$ are two set flows which are initially
disjoint, then they stay disjoint.
\end{prop}

\begin{prop}
\label{prop ncM uniqueness} If $\Omega_{0}$ is bounded and has a
strictly mean convex $C^{2}$-boundary, then there is a unique domain
flow $(\Omega_{t})_{t\in[0,\infty)}$ starting from $\Omega_{0}$.
It satisfies $\Omega_{t_{1}}\Supset\Omega_{t_{2}}$ for $t_{1}<t_{2}$.
\end{prop}

\begin{proof}
For some time, the flow $(\partial\Omega_{t})_{t}$ is smooth before
singularities form. For this time, any weak flow is unique. By Proposition
\ref{prop ncM weak avoidance}, for any $\varepsilon>0$ smaller then
the first singular time, $\Omega_{t-\varepsilon}\Supset\Omega_{t}\Supset\Omega_{t+\varepsilon}$
holds for $t\ge\varepsilon$. Any other weak solution is strictly
contained between $\Omega_{t\pm\varepsilon}$ for $t\ge\varepsilon$
as well. Note that $\varepsilon>0$ is arbitrary. Let $(\Psi_{t})_{t}$
be another weak solution. Let $t>0$. If $x\in\Psi_{t}$ then, by
openness, there is a closed ball around $x$ that is completely contained
inside $\Psi_{t}$. It takes some time $\tau$ to shrink that ball
to half its radius. Hence, $x\in\Psi_{t+\tau}\Subset\Omega_{t}$.
Since $x\in\Psi_{t}$ was arbitrary, $\Psi_{t}\subset\Omega_{t}$
follows. The same argument can be made to show the reverse inclusion.
We have just shown $\Psi_{t}=\Omega_{t}$ for arbitrary $t$.
\end{proof}
\begin{defn}
A closed subset $K_{0}\subset\R^{n+1}$ is said to be \emph{mean convex}
if $K_{t_{1}}\supset K_{t_{2}}$ for $t_{1}\le t_{2}$, where $(K_{t})_{t}$
denotes the level-set flow starting from $K_{0}$.
\end{defn}

\begin{rem*}
Adopting this definition, it is obvious that mean convexity is preserved
along the level-set flow. 
\end{rem*}
\begin{defn}
Let $K\subset\R^{n+1}$ be a closed subset. We define \emph{the mean
curvature in the viscosity sense in points $x\in\partial K$} by 
\begin{equation}
H(x)=\inf\big\{ H_{\partial A}(x)\colon A\subset K\text{ is a smooth domain and }x\in\partial A\big\}\;.
\end{equation}
The mean curvature in the viscosity sense may be infinite.
\end{defn}

\begin{defn}[$\alpha$-Noncollapsedness]
\label{def alpha-nc} Let $\alpha>0$. A closed subset of $\R^{n+1}$
is called \emph{$\alpha$-noncollapsed} if for any point $x\in\partial K$
the (viscosity) mean curvature satisfies $H(x)\in[0,\infty]$ and
there exist closed balls $\overline{B}_{\mathrm{int}}$ and $\overline{B}_{\mathrm{ext}}$
of radius $r(x)=\frac{\alpha}{H(x)}$ that contain $x$ and such that
$\overline{B}_{\mathrm{int}}\subset K$ and $\overline{B}_{\mathrm{ext}}\subset\R^{n+1}\setminus\mathrm{Int}(K)$
(see Fig.\ \ref{fig noncollapsed}).

We also say that a mean convex hypersurface $M_{t}$ is $\alpha$-noncollapsed
if the bounded closed region it bounds is $\alpha$-noncollapsed in
the above sense.

A family $(K_{t})_{t\in I}$ is called ${\alpha}$-noncollapsed if
$K_{t}$ is ${\alpha}$-noncollapsed for all $t\in I$.
\end{defn}

\begin{figure}
\label{fig noncollapsed} \centering \includegraphics[scale=0.4]{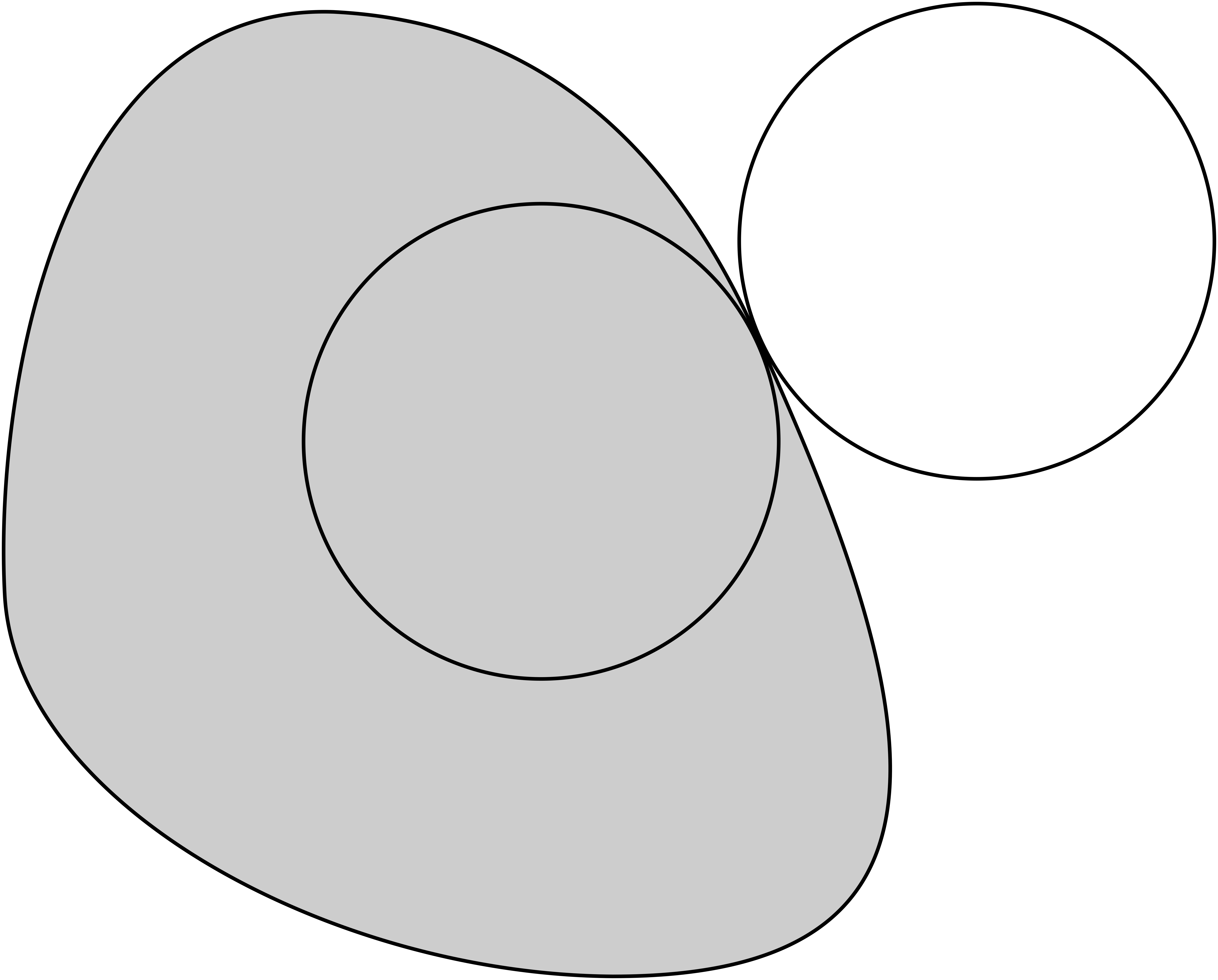}
\caption{Exterior and interior touching balls.}
\end{figure}

Smooth closed hypersurfaces with positive mean curvature $H>0$ are
$\alpha$-non\-col\-lapsed for some $\alpha>0$. B.~Andrews has
shown in \cite{And} that the hypersurface stays $\alpha$-non\-col\-lapsed
with the same $\alpha$ if one lets flow the hypersurface by its mean
curvature. This result holds even in a weak setting:
\begin{thm}
\label{thm ncMCF level-set flow alpha-noncollapsed} Let $K_{0}\subset\R^{n+1}$
be a compact, smooth, and mean convex domain that is $\alpha$-noncollapsed
for $\alpha>0$. Then the level-set flow that starts from $K_{0}$
is $\alpha$-noncollapsed. 
\end{thm}

\begin{thm}
\label{thm haslhofers curvature bound} For any $\alpha>0$, there
are $\rho=\rho(\alpha)>0$ and $C_{l}=C_{l}(\alpha)$ ($l=0,1,2,\ldots$)
with the following property. If $(M_{t})_{t\in I}$ is an $\alpha$-noncollapsed
mean curvature flow in a parabolic ball $P(x,t,r)\subset\R^{n+1}\times\R$
with $x\in M_{t}$ and $H(x,t)\le r^{-1}$, then 
\begin{equation}
\sup_{P(x,t,\rho\,r)}|\nabla^{l}A|\le C_{l}\,r^{-(l+1)}\;.
\end{equation}
\end{thm}

\begin{proof}[Idea of the proof]
 The theorem is proven with a blow-up argument. One considers a rescaled
sequence of counterexamples with $\sup|A|\ge1$ and $H(0)\to0$. Using
the $\alpha$-noncollapsedness, one can show that the sequence converges
locally smoothly to a hyperplane. This is the halfspace convergence
result of \cite{HK}. Contradiction with $\sup|A|\ge1$. 
\end{proof}
\begin{cor}
\label{cor alpha-nc curvature bound} For an $\alpha$-noncollapsed
mean curvature flow $(M_{t})_{t\in(0,T)}$, there are constants $C_{l}=C_{l}(\alpha)$
($l=0,1,2,\ldots$) such that 
\begin{equation}
|\nabla^{l}A|^{2}(x,t)\le C_{l}\,\big(t^{-1}+H^{2}(x,t)\big)^{l+1}\;.
\end{equation}
In particular, for $t_{0}>0$ there are constants $C_{l}=C_{l}(\alpha,t_{0})$
such that 
\begin{equation}
\frac{|\nabla^{l}A|^{2}}{(1+H^{2})^{l+1}}\le C_{l}(\alpha,t_{0})\qquad\text{holds for }t\ge t_{0}\;.
\end{equation}
\end{cor}

\begin{proof}
For $x\in M_{t}$ we set $r\coloneqq\big(t^{-1}+H^{2}(x,t)\big)^{-1/2}$.
Then there hold $r\le H(x,t)^{-1}$ and $r\le t^{1/2}$, which implies
$P(x,t,r)\subset\R^{n+1}\times(0,T)$. By Theorem \ref{thm haslhofers curvature bound}
there are constants $C_{l}(\alpha)$ such that $|\nabla^{l}A|(x,t)\le C_{l}\,r^{-(l+1)}=C_{l}\,\big(t^{-1}+H^{2}(x,t)\big)^{\frac{l+1}{2}}$
($l=0,1,\ldots$) holds. The assertion follows with $C_{l}^{2}$ instead
of $C_{l}$, which is just as fine.
\end{proof}

\lyxaddress{\begin{center}
Wolfgang A.\ Maurer, Fachbereich Mathematik und Statistik, Universität
Konstanz, 78457 Konstanz, Germany\\
e-mail: wolfgang.maurer@uni-konstanz.de
\par\end{center}}

\begin{thebibliography}{10}
\bibitem{And}\emph{Ben Andrews}, Noncollapsing in mean -convex mean
curvature flow, Geom. Topol. \textbf{16} (2012), 1413--1418.

\bibitem{CY}\emph{Bing-Long Chen} and \emph{Le Yin}, Uniqeuness and
pseudolocality theorems of the mean curvature flow, Comm.\ Anal.\ Geom.\ \textbf{15.3}
(2007), 435--490.

\bibitem{CGG}\emph{Yun Gang Chen}, \emph{Yoshikazu Giga}, and \emph{Shun'ichi
Goto}, Uniqueness and existence of viscosity solutions of generalized
mean curvature flow equations, J.\ Differential Geom.\ \textbf{33.3}
(1991), 749--786.

\bibitem{CM}\emph{Tobias H.\ Colding} and \emph{William P.\ Minicozzi,
II.}, Differentiability of the arrival time, arXiv e-prints (2015),
\texttt{arXiv:1501.07899}\textsf{.}

\bibitem{EH}\emph{Klaus Ecker} and \emph{Gerhard Huisken}, Interior
estimates for hypersurfaces moving by mean curvature, Invent.\ Math.\ \textbf{105.3}
(1991), 547--569.

\bibitem{ES}\emph{Lawrence C.\ Evans} and \emph{Joel Spruck}, Motion
of level sets by mean curvature. I, J.\ Differential Geom.\ \textbf{33.3}
(1991), 635--681.

\bibitem{GH}\emph{Michael Gage} and \emph{Richard S.\ Hamilton},
The heat equations shrinking convex plane curves, J.\ Differential
Geom.\  \textbf{23.1} (1986), 69--96.

\bibitem{Gra}\emph{Matthew A.\ Grayson}, The heat equation shrinks
embedded plane curves to round points, J.\ Differential Geom.\ 
\textbf{26.2} (1987), 285--314.

\bibitem{HK}\emph{Robert Haslhofer} and \emph{Bruce Kleiner}, Mean
Curvature Flow of Mean Convex Hypersurfaces, Comm Pure Appl. Math.
\textbf{70.3} (2016), 511--546.

\bibitem{Hui}\emph{Gerhard Huisken}, Nonparametric mean curvature
evolution with boundary conditions, J.\ Differential Equat.\  \textbf{77.2}
(1989), 369--378.

\bibitem{Ilm1}\emph{Tom Ilmanen}, The Level-set Flow on a Manifold,
Proc.\ Symp.\ Pure\ Math.\ \textbf{54.1} (1993), 193--204.

\bibitem{Ilm2}\emph{Tom Ilmanen}, Elliptic regularization and partial
regularity for motion by mean curvature, Mem.\ Amer.\ Math.\ Soc.\ \textbf{108.520}
(1994).

\bibitem{INS}\emph{Tom Ilmanen}, \emph{André Neves}, and \emph{Felix
Schulze}, On short time existence for the planar network flow, J.\ Differential
Geom.\ \textbf{111.1} (2019), 39--89.

\bibitem{IW}\emph{James Isenberg} and\emph{ Haotian Wu}, Mean curvature
flow of noncompact hypersurfaces with Type-II curvature blow-up, J.\ Reine
Angew.\ Math.\ \textbf{754} (2019), 225--251.

\bibitem{IWZ}\emph{James Isenberg}, \emph{Haotian Wu}, and \emph{Zhou
Zhang}, Mean curvature flow of non-compact hypersurfaces with Type-II
curvature blow-up. II, arXiv e-Prints (2019), \texttt{arXiv:1911.07282}\textsf{.}

\bibitem{Lah}\emph{Ananda Lahiri}, Almost graphical hypersurfaces
become graphical under mean curvature flow, arXiv e-Prints (2015),
\texttt{arXiv:1505.00543}\textsf{.}

\bibitem{Mau1}\emph{Wolfgang Maurer}, Shadows of graphical mean curvature
flow, Comm. Anal. Geom. \textbf{29.1} (2021), 183--206

\bibitem{Mau2}\emph{Wolfgang Maurer},\emph{ }Mean curvature flow
of symmetric double graphs only develops singularities on the hyperplane
of symmetry, arXiv e-Prints (2021), \texttt{arXiv:2103.06072}.

\bibitem{SS}\emph{Mariel Sáez Trumper} and \emph{Oliver Schnürer},
Mean curvature flow without singularities, J.\ Differential Geom.\ \textbf{97.3}
(2014), 545--570.
\end{thebibliography}
\end{document}